\documentclass[12pt]{article}

\usepackage[T1]{fontenc}
\usepackage{amsmath}
\usepackage{amssymb}
\usepackage{amsthm}
\usepackage{setspace}
\usepackage{graphicx}
\usepackage[margin=2.54cm]{geometry}
\usepackage{xcolor}
\usepackage[pdftex,colorlinks,linkcolor=black,citecolor=black,urlcolor=black]{hyperref}
\usepackage[size=\scriptsize]{todonotes}
\usepackage{natbib}
\usepackage{lmodern}
\usepackage{arydshln} 

\newtheorem{lemma}{Lemma}
\newtheorem{proposition}{Proposition}
\newtheorem{theorem}{Theorem}
\newtheorem{corollary}{Corollary}
\newtheorem{example}{Example}
\newtheorem{remark}{Remark}

\DeclareMathOperator{\st}{s. \hspace*{-0.03in} t.}
\DeclareMathOperator{\diag}{diag}
\DeclareMathOperator{\Diag}{Diag}
\DeclareMathOperator{\rank}{rank}
\newcommand{\mb}{\mathbf}
\def\texitem#1{\par\smallskip\noindent\hangindent 25pt
               \hbox to 25pt {\hss #1 ~}\ignorespaces}

\newcommand{\eps}{\epsilon}

\newcommand{\Cc}{\mathcal{C}}

\newcommand{\Jc}{\mathcal{J}}
\newcommand{\Kc}{\mathcal{K}}
\newcommand{\Pc}{\mathcal{P}}
\newcommand{\Sc}{\mathcal{S}}
\newcommand{\Lc}{\mathcal{L}}

\title{On the Semidefinite Representability of Continuous Quadratic Submodular Minimization \\ With Applications to \textcolor{black}{Pricing and} Moment Problems}
\author{Samuel Burer\thanks{Department of Management Sciences, University of Iowa, Iowa City, IA 52242-1994, USA. Email: samuel-burer@uiowa.edu}  \and Karthik Natarajan\thanks{Engineering Systems and Design, Singapore University of Technology and Design, 8 Somapah Road, Singapore 487372. Email: karthik\_natarajan@sutd.edu.sg} \and Rick Willemsen\thanks{Engineering Systems and Design, Singapore University of Technology and Design, 8 Somapah Road, Singapore 487372. Email: rick\_willemsen@sutd.edu.sg} }
\date{March 2026; \textcolor{black}{Revised September 2026}}

\begin{document}

\begin{onehalfspace}

\maketitle
\begin{abstract}
We study continuous quadratic submodular minimization with bounds and
\textcolor{black}{provide} a polynomially sized semidefinite relaxation,
which is provably tight for dimension $n \le 3$. \textcolor{black}{Via
an explicit counterexample for $n = 4$, we show that the relaxation
is not tight in general, although it remains empirically tight
on randomly generated instances.} We apply the relaxation to
\textcolor{black}{multi-product pricing and} two moment problems
arising in distributionally robust optimization and the computation of
covariance bounds. Accordingly, this research advances the ongoing study
of continuous submodular minimization and opens new application areas
therein.
\end{abstract}

\section{Introduction}\label{sec:Intro}

Submodular functions have long played an important role in discrete optimization \citep{Lovasz1983}. In recent years, their relevance for continuous optimization has also grown, for example, in the fields of machine learning and artificial intelligence \citep{Bach_2019,bilmes2022submodularitymachinelearningartificial, NIPS2017_58238e9a, JMLR:v23:21-0166, NEURIPS2019_b43a6403, pmlr-v70-staib17a}. A continuous function $f$ defined on $\mathbb{R}^n$ is said to be {\em submodular\/} when
\[
f(x) + f(y) \ge f(x \vee y) + f(x \wedge y),
\]
for all pairs of vectors $x,y \in \mathbb{R}^n$, where the $\vee$ and $\wedge$ operators calculate component-wise maximums and minimums, respectively. If $f$ is twice differentiable, submodularity is equivalent to all mixed second partial derivatives being nonpositive. When the pure second partials are also nonpositive, $f$ is said to be {\em DR-submodular\/}, where {\em DR\/} stands for {\em diminishing returns\/}.

Generally speaking, studies on continuous submodular optimization fall into two types: those addressing the maximization of submodular functions, and those addressing minimization. Relatively more attention has been paid recently to maximization, as noted by \cite{Yu_Kucukyavuz_2024}, and we refer the reader to this paper for an excellent, recent summary of the literature on optimization with submodular functions. In our paper, we focus on {\em continuous submodular minimization (CSM)\/}.

As discussed by \cite{Bach_2019} and \cite{Axelrod_Liu_Sidford_2020}, many of the results for submodular minimization in discrete settings---in particular, minimization over the lattice $\{0,1\}^n$---have natural extensions to CSM over the box $[0,1]^n$. Indeed, the property of submodularity is generally considered to make minimization  easier due to a strong connection with convexity. Even still, a number of fundamental questions remain.

In this paper, we specifically study the minimization of a submodular quadratic function over $[0,1]^n$, a problem which we call {\em quadratic submodular minimization over the box (QSMB)\/}:
\begin{align}
\min \left\{x^T Q x + c^T x : x \in [0,1]^n \right\}, \label{equ:qpb}
\end{align}
where $Q$ is a symmetric matrix and $c$ is a column vector. We will
often refer to (\ref{equ:qpb}) by its pair $(Q,c)$. Here,
submodularity is equivalent to the off-diagonal entries of $Q$ being
nonpositive, and we say that such a $Q$ is a {\em submodular matrix\/}
and that the pair $(Q,c)$ is {\em submodular\/}. Note that, if the
constraints were $x \in [l, u]$, then we could convert to the form
(\ref{equ:qpb}) by a simple affine transformation of $[l,u]$ to
$[0,1]^n$ with a corresponding update to the objective function. This
transformation preserves certain properties of the Hessian matrix,
e.g., the signs of its eigenvalues and its submodularity. Another
transformation which preserves submodularity is the flip operation $x
\mapsto -x$.

When $Q$ is completely arbitrary, problem (\ref{equ:qpb}) is NP-hard
\citep{Horst.etal.2000}, but it is polynomial-time solvable in
certain cases---for example, when $Q$ is positive semidefinite and
hence (\ref{equ:qpb}) is convex. Submodularity also impacts the
computational complexity. In particular, when $(Q,c)$ is submodular
and also the diagonal entries of $Q$ are nonpositive, which is
DR-submodularity as defined above, then the objective function is
concave along each dimension and hence (\ref{equ:qpb}) reduces to
a special instance of submodular minimization over the lattice
$\{0,1\}^n$, which is well-known to be solvable in polynomial time; see
Proposition 1.1 in \cite{pmlr-v70-staib17a}. Problem (\ref{equ:qpb})
with DR-submodularity is in fact solvable using a polynomial sized
linear program; see Proposition 10 in \cite{Padberg}. Another important
case occurs when $(Q,c)$ is submodular and $c \le 0$. In this case,
\cite{Kim-Kojima_2003} build on the results of \cite{Zhang_2000} to
show that (\ref{equ:qpb}) can be solved in \textcolor{black}{polynomial
time} via its basic SDP relaxation; see Section \ref{sec:tight} below.
In addition, \cite{Axelrod_Liu_Sidford_2020} build on results in
\cite{Bach_2019} to present a probabilistic algorithm for solving CSM
over the box that includes QSMB. Their algorithm is linear in $n$ and
polynomial in the ratio $L/\varepsilon$, where $L$ is the submodular
function's Lipschitz parameter and $\varepsilon$ is the additive
error in the final optimal value. In particular, the polynomial
dependence on $L/\varepsilon$ makes their algorithm pseudo-polynomial.

\textcolor{black}{Table \ref{tab:complexity} provides a classification
of the complexity of the minimization of quadratic functions versus
the type of domain. Note that {\em QUBO\/} stands for {\em quadratic
unconstrained binary optimization\/}, which is standard in the
literature. In particular, this paper makes progress on quadratic
submodular minimization over the continuous domain using an SDP-based
approach.}

\begin{table}[!hbtp]
\centering
\textcolor{black}{%
\begin{tabular}{c||cc}
& $f$ convex & $f$ submodular\\\hline\hline
$\{0,1\}^n$ & NP-hard (QUBO) & P (LP) \\
            & \cite{Garey}\footnotemark& \cite{Padberg}\\ \hline
            $[0,1]^n$& P (convex QP) & {P (LP, when DR-submodular)}  \\
& \cite{Khachiyan} & \cite{Padberg} \\ \hdashline
& & {P (SDP, when $c \leq 0$)} \\
&  & \cite{Kim-Kojima_2003} \\ \hdashline
&  & {P (SDP, when $n\le3$)\footnotemark}  \\
& & {This paper}
\end{tabular}
\caption{Complexity of minimizing $f(x) = x^T Q x +  c^T x$}
\label{tab:complexity}
}
\end{table}
\addtocounter{footnote}{-2}
\stepcounter{footnote}\footnotetext{\textcolor{black}{Over $\{0,1\}^n$, $x^T Q x + c^T x = x^T
(Q-\lambda_{\min}I) x + (c+\lambda_{\min} e)^T x$ from the identity $x_i^2= x_i$ where $\lambda_{\min}$ is the minimum eigenvalue of $Q$. Since $Q-\lambda_{\min}I$ is positive semidefinite, the hardness of minimizing a convex quadratic function over $\{0,1\}^n$ follows directly from the hardness of QUBO.}}
\stepcounter{footnote}\footnotetext{\textcolor{black}{As we will show, tightness of the semidefinite relaxation is specific to $n \le 3$. In particular, Example \ref{exa:counter} exhibits a submodular instance with $n = 4$ for which the relaxation has a strictly positive gap. The complexity of minimizing a submodular quadratic over $[0,1]^n$ for $n \ge 4$ remains open.}}

\subsection{Motivating examples}

We consider a few examples to motivate the applicability and study of QSMB.

\begin{example}[Multi-product pricing] Consider $n$ products with the
price decision vector $p \in \mathbb{R}^n$, \textcolor{black}{marginal
cost vector $m \in \mathbb{R}^n$}, and linear demand model given by
$d(p) = a - B p$ where $a \in \mathbb{R}^n$ and $B \in \mathbb{R}^{n
\times n}$. The price vector $p$ is chosen in the box $[l,u]$ where
we assume $d(p) \geq 0$ for all $p \in [l,u]$. Then the multi-product
pricing problem is formulated as the quadratic maximization problem
\begin{equation*} \textcolor{black}{\max \left\{(p-m)^T(a-Bp): p \in
[l,u] \right\}.} \end{equation*} \textcolor{black}{In this model,
depending on the sign of the off-diagonal entries in the matrix $B$,
the products might be substitutes or complements. We particularly
focus on the pricing of substitute products such as grocery brands,
similar electronic products, competing package sizes, or differentiated
service plans. Demand models of substitutes require the off-diagonal
entries of $B$ to be nonpositive, so that the demand of product $i$
increases as the price of product $j$ increases for $j \neq i$. The
pricing problem with linear demand model and substitutable products
is then an instance of QSMB, where we symmetrize the quadratic term
by replacing $B$ with $(B + B^T)/2$ and minimize the negative of
the profit. In the literature, tractability is ensured by assuming
additionally that $(B + B^T)/2$ is positive semidefinite, so that the
pricing problem becomes a convex quadratic program. See Theorem 8.23 and
related demand models in Chapter 8 of \cite{Revenue} for multi-product
pricing with substitutes. A second route to tractability is to
discretize. \cite{NIPS2016_cb79f8fa} restrict each price to a finite
set of candidate values and encode that choice with binary variables;
substitutability then makes the gross profit supermodular, and the
pricing problem becomes a submodular binary quadratic minimization. We
employ neither convexity nor discretization, in which case solving the
pricing problem to optimality requires the solution of QSMB.}

\end{example}

\begin{example}[Ambiguity sets in distributionally robust optimization]
\textcolor{black}{A commonly used moment ambiguity set in distributionally robust optimization is the following Chebyshev ambiguity set (see \cite{bertsimas}, \cite{Dro}):
\begin{equation*}
\begin{array}{lll}
   \left\{ \mathbb{P} \in {\cal P}(\Xi) \ : \ \mathbb{E}_{\mathbb{P}}[\tilde{{\xi}}] = {\mu}, \ \mathbb{E}_{\mathbb{P}}[\tilde{{\xi}}\tilde{{\xi}}^T] = \Sigma \right\},
 \end{array}
 \end{equation*}
where the support of the random vector $\tilde{\xi}$ is contained in $\Xi$, the mean of the random vector is fixed to $\mu$,  and the second moment matrix is fixed to $\Sigma$. The set ${\cal P}(\Xi)$ denotes the set of all probability distributions with support contained in $\Xi$.  While the Chebyshev ambiguity set is computationally tractable for specific supports---such as when $\Xi$ equals $\mathbb{R}^n$ or is an ellipsoid---checking whether this ambiguity set is nonempty for $\Xi = [0,1]^n$ is known to be NP-hard in general.\footnotemark\footnotetext{\textcolor{black}{By setting the diagonal entries of $\Sigma$ equal to those of $\mu$, the random variables in the ambiguity set are forced to be Bernoulli when $\Xi = [0,1]^n$. Checking if the ambiguity set is nonempty is then equivalent to testing membership in the correlation polytope; see \cite{pitowsky}.}} A tractable relaxation of this ambiguity set proposed in \cite{Delage} is
\begin{equation*}
\begin{array}{lll}
   \left\{ \mathbb{P} \in {\cal P}([0,1]^n) \ : \ \mathbb{E}_{\mathbb{P}}[\tilde{{\xi}}] = {\mu}, \ \mathbb{E}_{\mathbb{P}}[\tilde{{\xi}}\tilde{{\xi}}^T] \preceq \Sigma \right\},
 \end{array}
 \end{equation*}
 where the exact second moment matrix condition is relaxed to an upper bound in the positive semidefinite order. In this paper, we will propose an alternative relaxation to the Chebyshev ambiguity set given by
 \begin{equation*}
\begin{array}{lll}
   \left\{ \mathbb{P} \in {\cal P}([0,1]^n) \ : \ \mathbb{E}_{\mathbb{P}}[\tilde{{\xi}}] = {\mu}, \ \mathbb{E}_{\mathbb{P}}[\diag(\tilde{{\xi}}\tilde{{\xi}}^T)] = \diag(\Sigma),  \ \mathbb{E}_{\mathbb{P}}[\tilde{{\xi}}\tilde{{\xi}}^T] \geq \Sigma \right\},
 \end{array}
 \end{equation*}
 where the mean vector and the diagonal entries of the second moment matrix are exactly specified and the off-diagonal entries of the second moment matrix are constrained by entrywise lower bounds. Using duality between the theory of moments and nonnegative polynomials, we will show that QSMB is key to providing a characterization of this relaxed ambiguity set. We discuss the flexibility of modeling uncertainty with this ambiguity set and its application to distributionally robust optimization in Section \ref{sec:dro}.
 }
\end{example}

\begin{example}[Covariance bounds]
\textcolor{black}{
The Cauchy--Schwarz inequality provides a sharp upper bound on the covariance of a pair of random variables using only the variances of the random variables:
\begin{equation*}
\begin{array}{rllll}
\mbox{Cov}[\tilde{\xi}_1,\tilde{\xi}_2] \leq \sqrt{\mbox{Var}[\tilde{\xi}_1]\mbox{Var}[\tilde{\xi}_2]}.
\end{array}
\end{equation*} This bound can be sharpened when each random variable is known to lie in the interval $[0,1]$ and the means are fixed (see \cite{Cov}):
\begin{equation*}
\begin{array}{rllll}
\mbox{Cov}[\tilde{\xi}_1,\tilde{\xi}_2] \leq \min\left(\mathbb{E}[\tilde{\xi}_1](1-\mathbb{E}[\tilde{\xi}_2]),\mathbb{E}[\tilde{\xi}_2](1-\mathbb{E}[\tilde{\xi}_1]),\sqrt{\mbox{Var}[\tilde{\xi}_1]\mbox{Var}[\tilde{\xi}_2]}\right).
\end{array}
\end{equation*}
Again, by building on duality theory for moment problems and the connection of the dual formulation to QSMB, we will develop new covariance bounds for an arbitrary number of random variables, thus extending the above inequality.
}
\end{example}

\subsection{Contributions and structure of the paper}

The main contributions and the structure of the paper are as follows:

\texitem{(a)} In light of the existing literature, our first goal is
to study a polynomial-time semidefinite programming relaxation of
(\ref{equ:qpb}), which is inspired by \cite{Kim-Kojima_2003}. This
relaxation is {\em tight\/} (i.e., admits no relaxation gap) for $n
\le 3$, a result which we prove in Section \ref{sec:tight} making
use of technical results established in Section \ref{sec:techlemma}.
\textcolor{black}{We further show that this result is sharp by exhibiting
an explicit submodular instance with $n = 4$ that has a strictly
positive relaxation gap. Nonetheless, the relaxation is empirically
tight on a randomly generated panel of instances of varying sizes
reported in Section \ref{sec:numerical}.}

\texitem{(b)} We then apply the semidefinite programming reformulation
of QSMB to approximate---and in low dimensions, solve exactly---two
optimization models involving moments of probability measures.
\begin{itemize}
\item In the first model discussed in Section \ref{sec:dro}, we
consider a class of distributionally robust optimization (DRO)
problems over the new ambiguity set. 
Our SDP for QSMB enables us to approximate this class
of DRO problems efficiently for general $n$ providing an upper bound on the worst-case expected
cost with an exact reformulation for $n = 3$.  
\item In the second model discussed in Section \ref{sec:cov}, we
compute an upper bound on the nonnegative weighted sum of
covariances of random variables given only the means and the variances
of the random variables, again supported in $[0,1]^n$. This extends
the Cauchy--Schwarz inequality to bounded random variables in higher
dimensions. We note that the SDP relaxation provides the sharpest upper
bound possible in low dimensions ($n \le 3$). In addition, we provide a new proof for the $n = 2$ case and sharp closed form solutions in special cases for arbitrary number of random variables.

\end{itemize}

\texitem{(c)} In the numerical results in Section \ref{sec:numerical},
we present our empirical evidence for practical tightness of the relaxation
in higher dimensions. \textcolor{black}{We also provide examples in dimension
$n = 3$ for multi-product pricing and a distributionally robust
optimization problem arising in stratified sampling allocation
to illustrate the applicability of QSMB and the value of the SDP
formulation}. We then compute moment bounds on the energy function defined
by Laplacian matrices of graphs. Our results illustrate that the
bounds from our approach can be stronger than bounds from existing SDP
formulations such as that developed in \cite{Delage} for some of these
problems.

\mbox{}

\noindent We conclude in Section \ref{sec:conclusion}. Overall, this
paper extends polynomial-time results for convex minimization to
quadratic submodular minimization over the box (QSMB) via a provably
exact SDP relaxation for $n \le 3$, \textcolor{black}{shows by an
explicit counterexample that exactness fails for $n \ge 4$,} and
examines various applications of these results. As such, this paper lays
the groundwork for further study of continuous submodular minimization
and indeed can have important ramifications in areas such as robust and
distributionally robust optimization.

\subsection{Notations}

We use $\mathbb{R}^n$ to denote the space of real $n$-dimensional
vectors, $\mathbb{S}^n$ to denote the space of symmetric real $n
\times n$ matrices, and $\mathbb{R}^{m \times n}$ to denote the space
of real $m \times n$ matrices. The vector $x$ is a column vector by
default and $x^T$ is the transposed row vector. We use $e$ to denote
the vector of ones\textcolor{black}{, whose dimension is determined by the context
in which it appears}. We write $A \succeq 0$ ($A \succ 0$) to denote
that matrix $A$ is positive semidefinite (positive definite) and $A
\succeq B$ to denote $A-B \succeq 0$. The vector formed with the
diagonal entries of the matrix $A$ is given by $\mbox{diag}(A)$, and
the diagonal matrix formed with the entries of the vector $a$ is given
by $\mbox{Diag}(a)$. Given $A, B \in \mathbb{R}^{m \times n}$, we let
$A \bullet B = \mbox{trace}(A^TB)$. For vectors $x, y \in
\mathbb{R}^n$, we use $x \circ y$ to denote their componentwise
(Hadamard) product. We use $\sqrt{x}$ to denote a vector with the
square root of the entries of the vector $x$. A random vector is
denoted by $\tilde{\xi}$, and we use $\mathbb{E}_\mathbb{P}[\cdot]$ to denote
the expectation with respect to the distribution $\mathbb{P}$,
$\mbox{Var}[\tilde{\xi}_i]$ to denote the variance of $\tilde{\xi}_i$,
and $\mbox{Cov}[\tilde{\xi}_i,\tilde{\xi}_j]$ to denote the covariance
of $\tilde{\xi}_i$ and $\tilde{\xi}_j$.

\section{A Tight Semidefinite Relaxation for \texorpdfstring{$n \le 3$}{n <= 3}} \label{sec:tight}

In this section, we establish our main theoretical result, namely that
there exists a tight semidefinite relaxation of (\ref{equ:qpb}) for
dimension $n \le 3$.  We first review the relevant literature that
motivates our approach.

\subsection{Existing results}

\cite{Kim-Kojima_2003} proved the following result for QSMB for
general $n$ with additional restrictions on $c$:

\begin{proposition}[Theorem 3.1 in \cite{Kim-Kojima_2003}] \label{pro:KimKojima}
Let $n$ be arbitrary, and suppose $(Q,c)$ is submodular with $c
\le 0$. Then the optimal value of (\ref{equ:qpb}) equals the optimal
value of the SDP relaxation
\[
\min \left\{ Q \bullet X + c^T x : \diag(X) \le e, \ \begin{pmatrix} 1 & x^T \\ x & X \end{pmatrix} \succeq 0 \right\}.
\]
\end{proposition}

\noindent In fact, the authors show that, if $(x^*, X^*)$ is an optimal
solution of the SDP, then the vector $\sqrt{\diag(X^*)}$ is an optimal
solution of (\ref{equ:qpb}). Furthermore, the positive semidefiniteness
condition on the $(n+1)\times (n+1)$ matrix can be relaxed to positive
semidefiniteness of the $n(n+1)/2$ principal submatrices of size $2
\times 2$ without changing the conclusion of the theorem; see Theorem
3.5 in \cite{Kim-Kojima_2003}. It is also straightforward to construct
examples with $n=2$ and some $c_j > 0$ such that this SDP relaxation
admits a gap.

When $c$ is arbitrary, we can attempt to close the gap using a tighter
SDP relaxation. A typical approach, which we adapt in this paper, is
to include valid inequalities derived from the feasibility condition
$x \in [0,1]^n$. Indeed, a common class of valid inequalities are the
so-called {\em McCormick inequalities\/} or {\em RLT (reformulation
linearization technique) constraints\/}, which are derived for each
pair $1 \le i \le j \le n$ as follows \citep{McCormick.1976}:
\[
\begin{array}{rcl}
x_i x_j \ge 0 \quad & \Longrightarrow & \quad X_{ij} \ge 0 \\
x_i (1 - x_j) \ge 0 \quad & \Longrightarrow & \quad X_{ij} \le x_i \\
(1 - x_i)x_j \ge 0 \quad & \Longrightarrow & \quad X_{ij} \le x_j \\
(1 - x_i)(1 - x_j) \ge 0 \quad & \Longrightarrow & \quad X_{ij} \ge x_i + x_j - 1.
\end{array}
\]
\noindent In matrix form, these valid constraints are expressed as
\begin{equation*}
    X \ge 0, \quad\quad X \ge xe^T + ex^T - ee^T, \quad\quad X \le xe^T.
\end{equation*}

\noindent Note that two of these matrix inequalities provide lower
bounds on the components of $X$, whereas the third provides upper
bounds. In addition, the right-hand side of $X \le xe^T$ is
non-symmetric, and so this single inequality captures both $X_{ij} \le
x_i$ and $X_{ij} \le x_j$ for all $i \le j$. Also note that the
diagonal inequalities of $X \le xe^T$ ensure $\diag(X) \le x$, a
strengthening of the constraint $\diag(X) \le e$ in Proposition
\ref{pro:KimKojima}.

The SDP relaxation, which incorporates the RLT constraints, is known to
be tight for $n=2$ for all data inputs $(Q,c)$:

\begin{proposition}[Theorem 2 in \cite{Anstreicher.Burer.2010}] \label{pro:BurerAnstreicher}
    For $n=2$ and arbitrary $(Q,c)$, the SDP relaxation
    \begin{align*}
        \min \left\{ Q \bullet X + c^T x : 
        \begin{array}{l}
        X \ge 0, \\
        X \ge xe^T + ex^T - ee^T, \\
        X \le xe^T,
        \end{array} \
        \begin{pmatrix} 1 & x^T \\ x & X \end{pmatrix} \succeq 0 \right\}
    \end{align*}
    is a tight relaxation of (\ref{equ:qpb}).
\end{proposition}

\noindent However, the authors also show that this relaxation admits a
gap for $n \ge 3$ and general $(Q,c)$.

\subsection{Additional background}

Focusing our attention on the submodular case, we consider the
following relaxation, which adds only the RLT upper bounds to the
relaxation in Proposition \ref{pro:KimKojima}:
\begin{align}
    \min \left\{ Q \bullet X + c^T x : X \le xe^T, \ \begin{pmatrix} 1 & x^T \\ x & X \end{pmatrix} \succeq 0 \right\}. \label{equ:sdp}
\end{align}
The intuition for only including the upper bounds comes from the
submodularity of $Q$, which implies that the objective function is
nonincreasing in the off-diagonal entries of $X$ and hence
the RLT lower bounds are less likely to be active at optimality.

\textcolor{black}{Our main result in Theorem \ref{the:n=3} states
that this relaxation is tight when $n \le 3$, for every submodular
$(Q,c)$ and irrespective of the signs of the entries in $c$. Example
\ref{exa:counter} below shows that this cannot be extended; tightness
fails already at $n = 4$.}

\textcolor{black}{Two features of the relaxation~(\ref{equ:sdp}) are worth
emphasizing. First, compared to the full SDP-RLT relaxation of
Proposition~\ref{pro:BurerAnstreicher}, it retains roughly half the
number of linear inequalities. Second---and as a consequence---tightness
of~(\ref{equ:sdp}) as expressed in Theorem \ref{the:n=3} is a stronger
statement than tightness of the full SDP-RLT relaxation. Because the
latter is at least as tight, its exactness follows from the exactness
of~(\ref{equ:sdp}).}

We continue by introducing some notation and important concepts. Define
\[
    Y(x,X) := \begin{pmatrix} 1 & x^T \\ x & X \end{pmatrix}
\]
to be the symmetric matrix appearing in the SDP relaxation
(\ref{equ:sdp}). In particular, $Y(x,X) \succeq 0$ is the positive
semidefinite condition. Also define
\begin{align*}
    r(x,X) &:= \text{rank}(Y(x,X)), \\
    a(x,X) &:= \left| \left\{ (i,j) : X_{ij} = x_i \right\} \right|.
\end{align*}
In words, $a(x,X)$ is the number of active inequalities in the linear
inequality condition $X \le x e^T$ of (\ref{equ:sdp}). In particular,
$a(x,X) \le n^2$ since there are $n^2$ inequalities in $X \le xe^T$.

The following lemma relates $r(x,X)$ and $a(x,X)$ at extreme points of
(\ref{equ:sdp}). In particular, the number of active inequalities grows
quadratically with the rank.

\begin{lemma} \label{lem:pataki}
    Let $(x,X)$ be an extreme point of the feasible set of
    (\ref{equ:sdp}). Define $r := r(x,X)$ and $a := a(x,X)$. It
    holds that
    \[
        r(r + 1) \le 2(a + 1).
    \]
    Thus, irrespective of $n$, the following low-rank combinations of
    $r$ and $a$ are possible: $r = 1$ with $a \ge 0$; $r = 2$ with $a \ge
    2$; $r = 3$ with $a \ge 5$; and $r = 4$ with $a \ge 9$.
\end{lemma}

\begin{proof}
We recall a celebrated result due to \cite{Pataki.1998}:

\begin{quote}
Consider an SDP feasible set in block standard form: $F := \{ X^j \succeq 0,\ j
= 1,\ldots,p : \sum_{j=1}^p A^j_i \bullet X^j = b_i,\ i = 1,\ldots,m \}$. Let
$(X^1,\ldots,X^p)$ be an extreme point of $F$, and define $r_j := \rank(X^j)$.
Then $\sum_{j=1}^p r_j(r_j+1) \le 2m$.
\end{quote}

    \noindent The feasible set of (\ref{equ:sdp}) can be expressed in
    block standard form with $p = n^2 + 1$ blocks: one PSD block and
    $n^2$ slack-variables. Moreover, the number of equality constraints
    is $m = n^2 + 1$ with $n^2$ defining the slack variables and one
    defining the top-left entry of $Y(x,X)$ to be 1. The formula
    now follows by applying Pataki's result and noting that, for a
    slack-variable, $r_j (r_j + 1) = 0$ if and only if the corresponding
    slack variable equals 0 and $r_j (r_j + 1) = 2$ if and only if the
    slack variable is positive. In particular, if $s$ is the number of
    positive slacks such that $s + a = n^2$, then we have
    \[
    r(r + 1) + 2s \le 2(n^2 + 1) \quad \Leftrightarrow \quad r(r + 1) + 2(n^2 - a) \le 2(n^2+1)
    \quad \Leftrightarrow \quad r(r + 1) \le 2 (a + 1).
    \]
    The final statement of the lemma follows by considering 
    different cases for the rank of $Y(x,X)$.
\end{proof}

Next, we describe {\em completely positive programming\/}, which
provides a tool for proving tightness of SDP relaxations. Define the
cone
\[
    \Jc := \left\{ y := \begin{pmatrix} x_0 \\ x \end{pmatrix} \textcolor{black}{{}\in \mathbb{R}^{n+1}} : 0 \le x \le x_0 e \right\},
\]
which is the homogenization of the unit box $[0,1]^n$\textcolor{black}{, where $x_0
\in \mathbb{R}_+$ and $x \in \mathbb{R}^n$}, and the {\em
completely positive cone with respect to $\Jc$\/}:
\[
    \Cc\Pc(\Jc) := \text{conv} \left\{ yy^T : y \in \Jc \right\}.
\]
It is known that the optimal value of (\ref{equ:qpb}) equals that of
the completely positive program $\min \{ Q \bullet X + c^T x : Y(x,X)
\in \Cc\Pc(\Jc) \}$; see \cite{Burer.2009, Burer2012}.

The SDP relaxation (\ref{equ:sdp}) can be viewed as relaxing complete
positivity to $Y(x,X) \in \Lc$, where
\[
    \Lc := \left\{ Y := \begin{pmatrix} X_{00} & x^T \\ x & X
    \end{pmatrix} \succeq 0 : X \le xe^T \right\}
\]
is the homogenization of the feasible set of (\ref{equ:sdp}). The
following tightness criterion is used throughout the paper.

\begin{lemma} \label{lem:cp}
$\Cc\Pc(\Jc) \subseteq \Lc$.  Furthermore, if there exists an optimal solution
$(x,X)$ of (\ref{equ:sdp}) such that $Y(x,X) \in \Cc\Pc(\Jc)$, then the SDP
relaxation is tight.
\end{lemma}

Next, we classify the $n^2$ inequalities $X_{ij} \le x_i$ into three
types based on their position in the non-symmetric matrix $X - xe^T$:
{\em s-lower inequalities\/} ($i > j$), {\em s-upper inequalities\/}
($i < j$), and {\em diagonal inequalities\/} ($i = j$). The terms
{\em s-lower\/} and {\em s-upper\/} are shorthand for ``\textcolor{black}{strictly} lower'' and
``strictly upper.''

Finally, given a feasible solution $(x,X)$ of
(\ref{equ:sdp}), we consider permuting the rows and columns of
$Y(x,X)$ so that $x_1 \le \cdots \le x_n$, which we call {\em sorting
by $x$\/}. Sorting by $x$ preserves feasibility in $\Lc$, the rank
$r(x,X)$, the number of active inequalities $a(x,X)$, and complete
positivity with respect to $\Jc$.  When $(x,X)$ is sorted by $x$ and
$i < j$, feasibility with respect to $\Lc$ gives $X_{ij} \le x_i \le
x_j$. Hence an active s-upper inequality ($X_{ij} = x_i$) need not
force the s-lower inequality to be active, but an active s-lower
inequality ($X_{ij} = x_j$) forces $x_i = x_j$ and hence the s-upper
inequality to be active as well. A precise consequence of these
observations is recorded in Lemma \ref{lem:lower_active_duplicate} in
Section \ref{sec:techlemma}.

\subsection{Our result}

Theorem \ref{the:n=3} is stated and proved below by combining several
propositions, which are proved in Section \ref{sec:techlemma}.

The first proposition addresses optimal solutions of rank~2: either
the SDP relaxation is tight, or a proper face of the optimal face
exists whose extreme points all have rank~3 or higher.

\begin{proposition} \label{pro:rank2}
If there exists an optimal solution $(x,X)$ of (\ref{equ:sdp}) with
$r(x,X) = 2$, then: (i) the SDP relaxation is tight; or (ii) there exists
a proper face of the optimal face in which all extreme points have rank 3 or
higher.
\end{proposition}

\noindent Our next proposition establishes tightness related to active
diagonal and s-upper inequalities.

\begin{proposition} \label{pro:diag}
Let $(x,X)$ be an optimal solution of (\ref{equ:sdp}). After
sorting by $x$, suppose one of the following holds:
\begin{itemize}
    \item $\diag(X) = x$;
    \item all s-upper inequalities are active, and all diagonal
    inequalities are active except for a single $X_{jj} < x_j$.
\end{itemize}
Then the SDP relaxation is tight.
\end{proposition}

\noindent The final proposition uses an inductive argument to handle
active s-lower inequalities at optimality.

\begin{proposition} \label{pro:at_least_one_strictly_lower}
Suppose the SDP relaxation (\ref{equ:sdp}) is tight when applied to
all submodular instances of (\ref{equ:qpb}) in $n-1$ variables. For
dimension $n$, let $(x,X)$ be an optimal solution of (\ref{equ:sdp})
for a submodular instance $(Q,c)$. After sorting by $x$, suppose there
exists at least one active s-lower inequality.  Then the SDP
relaxation is tight.
\end{proposition}

Using the three propositions above, we can now state and prove the main theorem.

\begin{theorem} \label{the:n=3}
For $n \le 3$, the SDP relaxation (\ref{equ:sdp}) is tight for all
submodular data $(Q,c)$.
\end{theorem}

\begin{proof}
Let $(x,X)$ be an optimal extreme point of (\ref{equ:sdp}), and define
$Y := Y(x,X)$, $r := r(x,X)$, and $a := a(x,X)$. We prove $n=1$,
$n=2$, and $n=3$ in turn.

When $n = 1$, the feasibility condition $X \le xe^T$ amounts to a
single inequality. Hence, $a \le 1$, and Lemma~\ref{lem:pataki}
implies $r = 1$.  It follows that $Y \in \Cc\Pc(\Jc)$. Then the
relaxation is tight by Lemma~\ref{lem:cp}.

When $n = 2$, there are $n^2 = 4$ inequalities in $X \le xe^T$.  By
Lemma~\ref{lem:pataki}, $r(r+1) \le 2(a+1) \le 10$, so $r \le 2$.  If
$r = 1$, the relaxation is tight as above. If $r = 2$,
Proposition~\ref{pro:rank2} implies either (i) the SDP relaxation is
tight or (ii) there exists a proper face of the optimal face whose
extreme points all have rank at least~3. Since any extreme point of a
face is an extreme point of the feasible set, the bound $r \le 2$
applies throughout, ruling out case~(ii). Hence the SDP relaxation is
tight.

When $n=3$, if $r=1$, then we are done. If $r=4$, then $a \ge 9$ by
Lemma \ref{lem:pataki}, so every inequality is active. In particular,
$\diag(X) = x$, and Proposition \ref{pro:diag} implies the SDP
relaxation is tight. If $r=2$, then by Proposition \ref{pro:rank2}
either the SDP relaxation is tight, or there exists a proper face of
the optimal face whose extreme points all have rank~3 or higher; any
such extreme point is handled by the $r = 3$ case next.

This leaves $r=3$, for which $a \ge 5$ by Lemma \ref{lem:pataki}.
Assume without loss of generality that $(x,X)$ is sorted by $x$. Let
$l$, $d$, and $u$ be the number of active s-lower, diagonal, and
s-upper inequalities, respectively. Then $l + d + u = a \ge 5$ with
$l, d, u \le 3$.  If $l \ge 1$, Proposition
\ref{pro:at_least_one_strictly_lower} applies directly. If $l = 0$,
then $d + u \ge 5$ with $d, u \le 3$, so $d \ge 2$. Either $d = 3$ and
$\diag(X) = x$, or $d = 2$ and $u = 3$, meaning all s-upper
inequalities and all but one diagonal inequality are active. In both
subcases, Proposition \ref{pro:diag} applies.
\end{proof}

\textcolor{black}{Figure \ref{fig:tightproj} illustrates the geometry of Theorem
\ref{the:n=3} for $n = 3$. Fixing the values of $x$ and $\diag(X)$
and treating the off-diagonals $(X_{12}, X_{13}, X_{23})$ as free, we
project three feasible sets onto the two-dimensional region defined by
the auxiliary variables $(X_{12}+X_{13},\, X_{13}+X_{23})$: (i) the
exact hull $\Cc\Pc(\Jc)$, (ii) the full SDP-RLT relaxation, and (iii)
our relaxation~(\ref{equ:sdp}). A submodular objective maximizes a
nonnegative combination of the off-diagonals, so its optimum lies on
the ``northeast'' boundary, precisely where all three bodies coincide,
making the relaxation tight even after the lower bounds on $X_{ij}$
are discarded. A non-submodular objective would instead point toward a
boundary where the bodies separate and hence a gap appears.}

\begin{figure}[!htbp]
\begin{center}
\textcolor{black}{\includegraphics[width=0.72\textwidth]{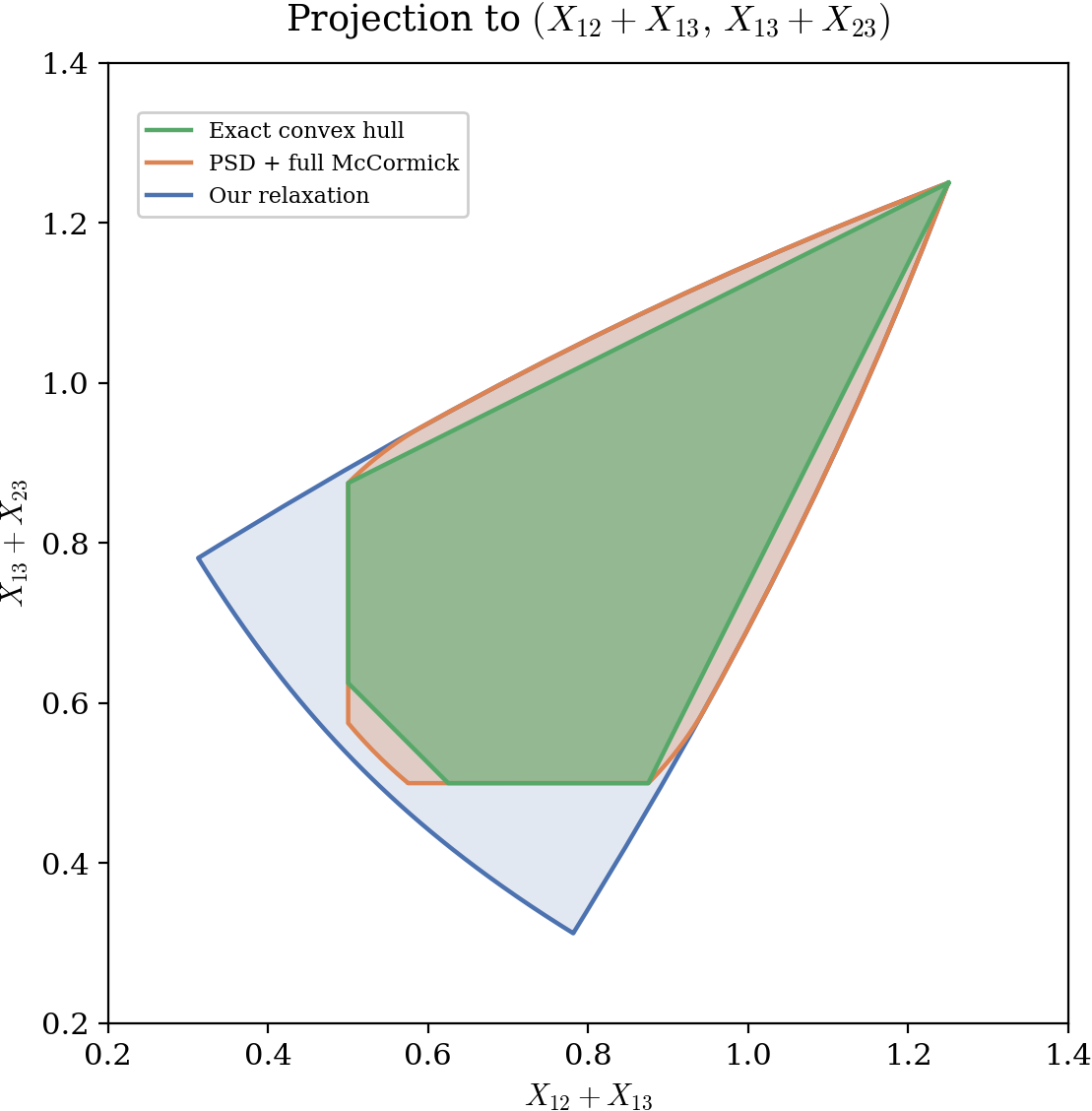}}
\caption{\textcolor{black}{Projection to $(X_{12}+X_{13},\,X_{13}+X_{23})$ of
three nested relaxations at fixed values of $x = \diag(X) = (5/8)e$:
the exact hull $\Cc\Pc(\Jc)$, the full SDP-RLT relaxation, and
our relaxation~(\ref{equ:sdp}). The three bodies coincide on the
``northeast'' boundary optimized by submodular objectives and do not
coincide elsewhere, where a non-submodular objective would incur a
gap.}}
\label{fig:tightproj}
\end{center}
\end{figure}

\begin{remark}[Extending beyond $n = 3$] \label{rem:n4}
    To extend Theorem \ref{the:n=3} to $n = 4$, {\color{black} in principle} the same proof strategy
could be applied. For $n = 4$, there are $n^2 = 16$ inequalities in $X
\le xe^T$, decomposed into $6$ s-lower, $4$ diagonal, and $6$ s-upper.
Lemma \ref{lem:pataki} constrains the rank to $r \le 5$. Among
the possible ranks, the cases $r = 1$, $r = 2$, $r = 4$, and $r = 5$
are resolved by the existing propositions together with counting
arguments. Specifically, Proposition \ref{pro:rank2} handles $r = 2$
by reducing to higher rank; for $r = 4$ and $r = 5$, the minimum
number of active inequalities ($a \ge 9$ and $a \ge 14$, respectively)
is large enough that either an s-lower inequality must be
active---invoking Proposition
\ref{pro:at_least_one_strictly_lower}---or the diagonal and s-upper
conditions of Proposition \ref{pro:diag} are met.

The bottleneck is $r = 3$, which requires only $a \ge 5$. When no
s-lower inequality is active ($l = 0$), one needs $d + u \ge 5$ with $d
\le 4$ and $u \le 6$. For $n = 3$, the constraints $d, u \le 3$ left
only two subcases---both covered by Proposition \ref{pro:diag}. For $n
= 4$, many additional configurations arise (e.g., $d = 2$, $u = 3$) in
which neither all diagonal nor all s-upper inequalities are active,
and these are not covered by the current propositions. Resolving these
rank-3 configurations with partial activity is the main obstacle to
\textcolor{black}{extending Theorem \ref{the:n=3} to $n = 4$ and beyond.
Example \ref{exa:counter} below shows that this obstacle does in fact
occur. We find that the relaxation (\ref{equ:sdp}) is {\em not\/} tight
for $n = 4$, and the counterexample occupies precisely the rank-3
configuration just described.}
\end{remark}

\begin{example}[The relaxation is not tight for $n \ge 4$] \label{exa:counter}
\textcolor{black}{Let $n = 4$ and consider the data
\begin{equation} \label{equ:counter}
Q = \begin{pmatrix}
    8 & -14 &   0 &   0 \\
  -14 &  25 & -25 &   0 \\
    0 & -25 &  25 & -14 \\
    0 &   0 & -14 &   8
\end{pmatrix},
\qquad
c = \begin{pmatrix} 12 \\ 29 \\ 0 \\ 0 \end{pmatrix}.
\end{equation}
Every off-diagonal entry of $Q$ is nonpositive, so $(Q,c)$ is submodular, and
$Q$ is indefinite. The box-constrained problem (\ref{equ:qpb}) has optimal value $0$,
attained uniquely at $x = 0$, while the relaxation (\ref{equ:sdp}) has optimal
value approximately $-0.1221$. So (\ref{equ:sdp}) has a strictly positive gap,
and Theorem \ref{the:n=3} does not extend to $n = 4$. Padding $(Q,c)$ with
additional variables having zero objective coefficients preserves both optimal
values, so the relaxation fails to be tight for every $n \ge 4$.}

\textcolor{black}{The optimal solution of (\ref{equ:sdp}) for this
instance has rank $3$, with exactly two diagonal and five s-upper
inequalities active and no s-lower inequality active. This is precisely
the configuration that Remark \ref{rem:n4} identifies as the obstacle
to extending the proof beyond $n = 3$. Section \ref{sec:counterexample}
verifies these claims in exact rational arithmetic and examines whether
standard cutting planes close the gap.}
\end{example}

\section{Distributionally Robust Optimization} \label{sec:dro}

Consider the distributionally robust optimization problem
\begin{equation} \label{dro}
\begin{array}{rlll}
\displaystyle \inf_{x \in {\cal X}}  \sup_{\mathbb{P} \in {\cal P}}\mathbb{E}_{\mathbb{P}}\left[f({x},\tilde{\xi}) := \max_{k=1,\ldots,K}\left(\tilde{{\xi}}^TA_k(x)\tilde{\xi}+{b}_k^T(x)\tilde{{\xi}} + c_{k}({x})\right)\right],
\end{array}
\end{equation}
where the decision vector $x$ is chosen from a set ${\cal X} \subseteq
\mathbb{R}^m$ and the random vector $\tilde{\xi}$ is $n$-dimensional.
The probability distribution of $\tilde{\xi}$, denoted by $\mathbb{P}$,
is ambiguous and lies in a set of probability distributions denoted
by ${\cal P}$, commonly referred to as the {\em ambiguity set\/}. The
matrix $A_k(x)\in \mathbb{S}^n$, the vector $b_k(x)\in \mathbb{R}^n$ and
the scalar $c_k(x) \in \mathbb{R}$ are assumed to depend affinely on
$x$ for each $k = 1,\ldots,K$. We hence assume that the cost function
$f({x},{\xi})$ is piecewise affine and thus convex in ${x}$ for a fixed
${\xi}$ and is piecewise quadratic but not necessarily convex in ${\xi}$
for a fixed ${x}$.

 Solving \eqref{dro} corresponds to finding the optimal decision $x$
 that minimizes the worst-case expected cost computed over all
 distributions in the ambiguity set. In this section, our goal is to
 propose a new moment-based ambiguity set, which incorporates
 submodular minimization and guarantees a computationally
 tractable---and empirically tight---bound via the SDP relaxation
 (\ref{equ:sdp}).

\subsection{A new moment ambiguity set}

 Given fixed $\mu \in \mathbb{R}^n$ and $\Sigma \in \mathbb{S}^{n}$, define
\begin{equation} \label{eq:prob}
\begin{array}{lll}
    {\cal P}  := \left\{
    \mathbb{P} \in {\cal P}([0,1]^n) \ : \
    \mathbb{E}_{\mathbb{P}}[\tilde{\xi}] = \mu, \
    \mathbb{E}_{\mathbb{P}}[\mbox{diag}(\tilde{\xi}\tilde{\xi}^T)] = \mbox{diag}(\Sigma), \
    \mathbb{E}_{\mathbb{P}}[\tilde{\xi}\tilde{\xi}^T] \geq \Sigma
    \right\}.
 \end{array}
 \end{equation}
%
\textcolor{black}{As discussed in the Introduction, this ambiguity
set can be viewed as a relaxation of the NP-hard Chebyshev ambiguity
set, which fixes the means and the diagonal entries of the second
moment matrix while relaxing the off-diagonal entries with entrywise
lower bounds. Note that the matrix $\Sigma$ is also not necessarily
required to be positive semidefinite, which contrasts with the Delage-Ye
relaxation also mentioned in the Introduction. This can be a modeling
advantage. For example, when the pairwise moments are estimated using
different subsets of observations or with missing entries scattered
across observations, the estimated entries of $\Sigma$ might not jointly
form a positive semidefinite matrix; see \cite{higham}.}


We would like to identify tractable necessary and sufficient conditions on $\mu$ and $\Sigma$ that guarantee nonemptiness of ${\cal P}$. Let ${\cal M}([0,1]^n)$ denote the set of all nonnegative finite Borel measures on $[0,1]^n$. As is standard, ${\cal P}$ is nonempty if and only if $(1,\mu,\Sigma) \in \mathbb{R} \times \mathbb{R}^n \times \mathbb{S}^n$ is a member of the following cone:
\begin{equation} \label{eq:mom}
\textcolor{black}{{\cal M}_0} :=
\left\{
    (\lambda_0,\lambda,\Lambda) \ : \
    \exists \ m \in {\cal M}([0,1]^n) \ \ \text{s.t.} \ \
    \begin{array}{l}
    \lambda_0 = \int dm(\xi), \\
    \lambda = \int{\xi} dm(\xi), \\
    \Lambda \le \int{\xi}{\xi}^T dm(\xi), \\
    \diag(\Lambda) = \int \diag({\xi}{\xi}^T) dm(\xi)
    \end{array}
\right\}.
\end{equation}
We will now argue that $\textcolor{black}{{\cal M}_0}$ is semidefinite approximable. To
this end, we start by defining some relevant convex cones.

Define the {\em truncated moment cone of degree 2\/}:
\begin{equation*}
{\cal M}_1 :=
\left\{
    (\lambda_0,\lambda,\Lambda) \in \mathbb{R} \times \mathbb{R}^n \times \mathbb{S}^n \ : \
    \exists \ m \in {\cal M}([0,1]^n) \ \ \text{s.t.} \ \
    \begin{array}{l}
    \lambda_0 = \int dm(\xi), \\
    \lambda = \int{\xi} dm(\xi), \\
    \Lambda = \int{\xi}{\xi}^T dm(\xi)
    \end{array}
\right\}.
 \end{equation*}
We also define the {\em cone of nonnegative polynomials of degree 2\/} on $[0,1]^n$ as
\begin{equation*}
     \Kc_1 := \left\{
     (y_0,y,Y) \in \mathbb{R} \times \mathbb{R}^n \times \mathbb{S}^n \ : \
     y_0 + y^T \xi + \xi^T Y \xi \geq 0 \ \ \forall \ \xi \in [0,1]^n
     \right\}.
 \end{equation*}
 The cones ${\cal M}_1$ and $\Kc_1$ are full dimensional, closed, convex, pointed and dual to each other; see \cite{studden,Monique,schmud}. Also define:
\begin{equation*}
{\cal M}_2 :=
\left\{
    (0,0,\Lambda) \in \mathbb{R} \times \mathbb{R}^n \times \mathbb{S}^n \ : \
    \Lambda \le 0,
    \diag(\Lambda) = 0
\right\}.
 \end{equation*}
and
\begin{equation*}
\begin{array}{rlll}
     \Kc_2 &:= \left\{
     (y_0,y,Y) \in \mathbb{R} \times \mathbb{R}^n \times \mathbb{S}^n \ : \
     Y_{ij} \leq 0 \ \ \forall \ i < j
     \right\} \nonumber \\
     &= \left\{ (y_0,y,Y) : Y \text{ submodular} \right\}.
     \end{array}
 \end{equation*}
 It is easy to check that the cones ${\cal M}_2$ and $\Kc_2$ are closed, convex and dual to each other.

It is now self-evident that $\textcolor{black}{{\cal M}_0} = {\cal M}_1 + {\cal M}_2$. Hence,
%
$\textcolor{black}{{\cal M}_0}$ is clearly a full-dimensional, convex, pointed cone.  
Closedness of $\textcolor{black}{{\cal M}_0}$ follows from Corollary 1.12 in
\cite{dickinson}.  The dual cone of $\textcolor{black}{{\cal M}_0}$ is given by
\begin{align}
\textcolor{black}{{\cal M}_0}^*  &=  \displaystyle \left({\cal M}_1 + {\cal M}_2\right)^*
 =  {\cal M}_1^* \cap {\cal M}_2^*
 =  \Kc_1 \cap \Kc_2 \nonumber \\
& = \left\{ (y_0,y,Y) \ : \ y_0+ y^T \xi + \xi^TY\xi \geq 0 \ \ \forall \ \xi \in [0,1]^n, \ \ Y \text{ submodular} \right\}
\\
& =: \textcolor{black}{{\cal K}_0}. \nonumber
 \end{align}
It follows that $\textcolor{black}{{\cal M}_0} = \textcolor{black}{{\cal K}_0^*}$ since $\textcolor{black}{{\cal M}_0}$ is a closed convex cone. 
This establishes the following proposition.
 \begin{proposition} \label{prop:duality}
The cones $\textcolor{black}{{\cal M}_0}$ and $\textcolor{black}{{\cal K}_0}$ are dual to each other.
\end{proposition}
 While the component cones ${\cal M}_1$ and $\Kc_1$ do not have tractable characterizations, we are nevertheless able to provide tractable semidefinite approximations of $\textcolor{black}{{\cal M}_0}$ and $\textcolor{black}{{\cal K}_0}$ using submodularity and the empirically tight SDP relaxation (\ref{equ:sdp}) of Section \ref{sec:tight}.
\begin{proposition} \label{prop:sdpmom}
Define
\[
\Kc_{\rm SDP} := \left\{ (y_0,y,Y) :
\begin{array}{l}
  Y \text{submodular}, \\
  \exists \ Z \in \mathbb{R}^{n \times n} \text{ s.t. } Z \ge 0, \
  \begin{pmatrix} y_0 &(y^T-e^TZ)/2\\ (y-Z^Te)/2 & Y + (Z+Z^T)/2 \end{pmatrix} \succeq 0
\end{array}
\right\}
\]
and
\[
{\cal M}_{\rm SDP} := \left\{ (\lambda_0,\lambda,\Lambda) \ : \
\exists \ W \in \mathbb{S}^n \text{ s.t. } \
\begin{array}{l}
  \Lambda \leq W \leq \lambda e^T, \\
  \diag(\Lambda) = \diag(W) ,\\
  \begin{pmatrix}
    \lambda_0 & \lambda^T\\
    \lambda & W
  \end{pmatrix} \succeq 0
\end{array}
\right\}.
\]
Then $\Kc_{\rm SDP} \subseteq \textcolor{black}{{\cal K}_0}$ and $\textcolor{black}{{\cal M}_0} \subseteq {\cal
M}_{\rm SDP}$ with equality in both when $n \le 3$.
\end{proposition}

\begin{proof}
Introducing an auxiliary variable $z_0$, we first observe that
\[
\Kc_{\rm SDP}
=
\left\{ (y_0,y,Y) \ : \
\begin{array}{l}
y_0 - z_0 \ge 0, \\
Z \geq 0, \ \begin{pmatrix} z_0 &(y^T-e^TZ)/2\\ (y-Z^Te)/2 & Y + (Z+Z^T)/2 \end{pmatrix} \succeq 0 \\
Y \text{ submodular}
\end{array}
\right\}.
\]
Letting $(y_0,y,Y,z_0)$ be an element of the right-hand side, to show
$(y_0,y,Y) \in \textcolor{black}{{\cal K}_0}$, we must demonstrate that $y_0 + y^T\xi +
\xi^TY\xi \ge 0$ for all $\xi \in [0,1]^n$, which is equivalent to
showing
\[
\nu := \min \left\{ y_0+y^T\xi+\xi^TY\xi :\xi \in [0,1]^n \right\} \ge 0,
\]
which in particular implies $y_0 \ge 0$. We have
\begin{align*}
\nu \ge \nu_{\rm SDP} &:=
\min\left\{y_0+y^T\xi+Y\bullet \Xi \ :\ \Xi \leq \xi e^T, \
\begin{pmatrix}1 &\xi^T\\ \xi & \Xi \end{pmatrix} \succeq 0 \right\} \\
&= \max\left\{y_0 - z_0 \ :\ Z \geq 0, \ \begin{pmatrix} z_0 &(y^T-e^TZ)/2\\ (y-Z^Te)/2 & Y + (Z+Z^T)/2 \end{pmatrix} \succeq 0 \right\},
\end{align*}
where strong duality holds, for example, by Lemma 18 in
\cite{Qiu-Yildirim_2024}. Furthermore, $\nu = \nu_{\rm SDP}$ holds
when $n \le 3$ by Theorem \ref{the:n=3}.  In particular, $\nu \ge
\nu_{\rm SDP} \ge 0$ holds when at least one dual solution has
nonnegative objective value, which is true by assumption. The expression for ${\cal M}_{\rm SDP}$
follows by standard dual constructions, giving $\textcolor{black}{{\cal M}_0} \subseteq
{\cal M}_{\rm SDP}$ with equality for $n \le 3$.
\end{proof}

\noindent Note that the auxiliary square matrix $Z \in \mathbb{R}^{n
\times n}$ used in $\Kc_{\rm SDP}$ is not symmetric, while $W$ in
${\cal M}_{\rm SDP}$ is symmetric.

We summarize the foregoing discussion by concluding that the ambiguity
set ${\cal P}$ in \eqref{eq:prob} is nonempty only if there exists a
matrix $W$ in $\mathbb{S}^{n}$ such that
 \begin{equation} \label{eq:momfeas}
\begin{array}{rlll}
 \displaystyle \Sigma \le W \leq \mu e^T, \quad \diag(W) = \diag(\Sigma), \quad
 \ \begin{pmatrix}
    1 & \mu^T\\
    \mu & W
  \end{pmatrix} \succeq 0.
 \end{array}
\end{equation}

\noindent For $n \le 3$, the converse also holds, so that
\eqref{eq:momfeas} is a necessary and sufficient condition for
nonemptiness of ${\cal P}$.

\subsection{Semidefinite approximation}

The following proposition shows that the optimal value of the
distributionally robust problem \eqref{dro} can be bounded above in
polynomial time by a semidefinite program when the ambiguity set
${\cal P}$ is given by \eqref{eq:prob}. Moreover, the bound is exact
when $n \le 3$.

\begin{proposition} \label{pr:dro}
Suppose:
\begin{enumerate}
    \item[(a)]  ${\cal X}$ is semidefinite representable;
    \item[(b)]  $-A_k(x)$ is submodular for each $x \in {\cal X}$ and each $k = 1,\ldots,K$.
\end{enumerate}
Then \eqref{dro} is bounded above in polynomial time by the SDP
\begin{equation} \label{primal1momsdp}
\begin{array}{rllll}
\displaystyle \inf & \displaystyle y_0 + \mu^Ty + \Sigma \bullet Y & \\
\st \;
& \displaystyle x \in {\cal X}\\
 & \displaystyle Y \text{ submodular} &\\
 & \displaystyle Z_k \geq 0 & \forall \ k =1,\ldots,K\\
& \displaystyle \begin{pmatrix} y_0 & (y^T-e^TZ_k)/2\\ (y-Z_k^Te)/2 & Y + (Z_k+Z_k^T)/2 \end{pmatrix} \succeq \begin{pmatrix} c_k(x) & b_k^T(x)/2 \\ b_k(x)/2 & A_k(x)\end{pmatrix}  & \forall \ k =1,\ldots,K,
\end{array}
\end{equation}
where the decision variables are $x \in \mathbb{R}^m$, $(y_0, y, Y) \in \mathbb{R} \times \mathbb{R}^n \times \mathbb{S}^n$, and $Z_k \in \mathbb{R}^{n\times n}$ for $k= 1,\ldots,K$ with equality holding for $n = 3$.
 \end{proposition}
 \begin{proof}
For fixed $x \in {\cal X}$, denote the value of the inner supremum in \eqref{dro} by
\begin{equation*}
\begin{array}{rllll}
v^*(x) := \sup \displaystyle \left\{\mathbb{E}_{\mathbb{P}}\left[\max_{k =1,\ldots,K}\left(\tilde{{\xi}}^T{A}_k(x)\tilde{{\xi}} + {b}_k^T(x)\tilde{{\xi}} + c_{k}({x})\right)\right] :  \mathbb{P} \in {\cal P}\right\}.
\end{array}
\end{equation*}
The dual is given by
\begin{equation*}
\begin{array}{rllll}
v_d^*(x)  = \displaystyle \inf & \displaystyle y_0 + \mu^Ty + Y \bullet \Sigma & \\
\st
 & \displaystyle Y \text{ submodular} \\
& \displaystyle y_0+y^T\xi +\xi^TY\xi \geq \max_{k =1,\ldots,K} \left(\xi^TA_k(x)\xi+{b}_k^T(x){{\xi}} + c_{k}({x})\right) &  \forall \ {\xi} \in [0,1]^n,
\end{array}
\end{equation*}
where $v^*(x) \leq v_d^*(x)$ holds from weak duality. 
Disaggregating the dual constraints across $k = 1,\ldots,K$ gives
 \begin{equation*}
\begin{array}{rllll}
v_d^*(x) = \displaystyle \inf & \displaystyle y_0 + \mu^Ty + \Sigma \bullet Y & \\
\st
 & \displaystyle Y \text{ submodular} \\
 & \displaystyle \left( y_0-c_k(x),y-b_k(x),Y-A_k(x) \right) \in \textcolor{black}{{\cal K}_0} & \forall \ k = 1,\ldots,K.
\end{array}
\end{equation*}
Here $Y-A_k(x)$ is submodular for each $k$ and $x$. Replacing $\textcolor{black}{{\cal K}_0}$ with the inner approximation $\Kc_{\rm SDP} \subseteq
 \textcolor{black}{{\cal K}_0}$ from Proposition \ref{prop:sdpmom} gives
 \begin{equation*}
\begin{array}{rllll}
\overline{v}_d^*(x) = \displaystyle \inf & \displaystyle y_0 + \mu^Ty + \Sigma \bullet Y & \\
\st
 & \displaystyle Y \text{ submodular} \\
 & \displaystyle \left( y_0-c_k(x),y-b_k(x),Y-A_k(x) \right) \in \Kc_{\rm SDP} & \forall \ k = 1,\ldots,K,
\end{array}
\end{equation*}
where $v^*(x) \leq v_d^*(x)  \leq \overline{v}_d^*(x)$. Optimizing over $x \in
 {\cal X}$, we obtain the SDP upper bound as claimed. Equality
holds for $n \le 3$ due to Theorem \ref{the:n=3} and Proposition \ref{prop:sdpmom}.
 \end{proof}

\subsection{Comparison with an existing moment ambiguity set} \label{sec:delageye}

A popular and tractable moment ambiguity set analyzed by \cite{Delage} is
\begin{equation} \label{eq:del}
\begin{array}{lll}
   {\cal Q} = \left\{ \mathbb{P} \in {\cal P}(\Xi) \ : \ \mathbb{E}_{\mathbb{P}}[\tilde{{\xi}}] = {\mu}, \ \mathbb{E}_{\mathbb{P}}[\tilde{{\xi}}\tilde{{\xi}}^T] \preceq \Sigma \right\}.
 \end{array}
 \end{equation}
Here, the support of the random vector $\tilde{\xi}$ is contained in $\Xi$ which is assumed to be convex and compact, the mean is fixed to $\mu$, and the second moment matrix is bounded from above by $\Sigma$ in the positive semidefinite order. It is well-known that ${\cal Q}$ is nonempty if and only if
 \begin{equation*}
\begin{array}{rlll}
 \displaystyle \mu \in \Xi, \ \begin{pmatrix}
    1 & \mu^T\\
    \mu & \Sigma
  \end{pmatrix} \succeq 0.
 \end{array}
\end{equation*}
It is interesting to compare these conditions with the analogous ones (\ref{eq:momfeas}) for our ambiguity set ${\cal P}$. These two sets express different phenomena, e.g., the matrix $\Sigma$ in \eqref{eq:del} needs to be positive semidefinite for ${\cal Q}$ to be nonempty but not for ${\cal P}$ to be nonempty. Furthermore, when $\Xi = [0,1]^n$ and the matrix $A_k(x)$ is negative semidefinite for each $k$, the distributionally robust optimization (\ref{dro})---where ${\cal P}$ is replaced by ${\cal Q}$---can be reformulated as an SDP using duality (see \cite{Delage}):
  \begin{align}
\displaystyle \min \quad & \displaystyle y_0 + \mu^Ty + \Sigma \bullet Y & \nonumber \\
\st \quad \;
& \displaystyle x \in {\cal X} \nonumber \\
& \displaystyle Y \succeq 0  \label{primal2momsdp} \\
& \displaystyle z_k,w_k \geq 0 & \forall \ k=1,\ldots,K \nonumber \\
& \displaystyle \begin{pmatrix} y_0-e^Tz_k & (y+z_k-w_k)^T/2\\ (y+z_k-w_k)/2 & Y  \end{pmatrix} \succeq \begin{pmatrix} c_k(x) & b_k^T(x)/2 \\ b_k(x)/2 & A_k(x)\end{pmatrix}  & \forall \ k=1,\ldots,K, \nonumber
\end{align}
where the decision variables are $x \in \mathbb{R}^m$, $(y_0, y,
Y) \in \mathbb{R} \times \mathbb{R}^n \times \mathbb{S}^n$, and
$(z_k, w_k) \in \mathbb{R}^{n} \times \mathbb{R}^{n}$ for $k=
1,\ldots,K$. We provide a numerical comparison of the formulations
(\ref{primal1momsdp}) and (\ref{primal2momsdp}) in Section
\ref{sec:numerical} \textcolor{black}{with two examples that arise in:
(i) distributionally robust stratified sampling allocation; and (ii)
moment bounds on the quadratic energy of Laplacian matrices where the
matrices $-A_k(x)$ are both positive semidefinite and submodular, which
allows direct comparison of the SDP relaxations}.

\section{Covariance Bounds} \label{sec:cov}

This section studies sharp upper bounds on covariance terms when only
first and second moments are known and the support is $[0,1]^n$. We
first formulate a moment problem and characterize feasibility, then
derive an SDP upper bound that is exact for $n \le 3$ under
submodularity. We then provide closed-form bounds and explicit
extremal constructions for $n=2$ and then for general $n$ under an
additional overlap condition on marginal moments.

\subsection{A moment problem and SDP relaxation}

Given only the mean and variance of two random variables
$\tilde{\xi}_1$ and $\tilde{\xi}_2$, with no restrictions on their
support, a well-known upper bound on the pairwise moment is given by:
\begin{equation}\label{cauchy}
\begin{array}{rllll}
\mathbb{E}[\tilde{\xi}_1\tilde{\xi}_2] \leq \mathbb{E}[\tilde{\xi}_1]\mathbb{E}[\tilde{\xi}_2]+\sqrt{\mbox{Var}[\tilde{\xi}_1]\mbox{Var}[\tilde{\xi}_2]}. 
\end{array}
\end{equation}
This bound is known to be best possible. We consider a multivariate
version of this problem and develop sharp covariance bounds for
bounded random variables where the means and variances are known.

For $i=1,\ldots,n$, suppose random variable $\tilde{\xi}_i$ has
support contained in $[0,1]$ with known mean $\mu_i $ and standard
deviation $\sigma_i$, or equivalently the variance is $\sigma_i^2$.
Let $(\mu,\sigma)$ represent the vector of means and standard
deviations of the random vector $\tilde{\xi}$. Given a matrix $A \in
\mathbb{S}^{n}$, consider the moment problem
\begin{equation}\label{momb}
\begin{array}{rllll}
v^* =  \displaystyle\max  \left\{\mathbb{E}_{\mathbb{P}}[\tilde{\xi}^TA\tilde{\xi}] \ : \
\begin{array}{c}
\mathbb{P} \in {\cal P}([0,1]^n), \
\mathbb{E}_{\mathbb{P}}[\tilde{\xi}] = \mu, \\
\mathbb{E}_{\mathbb{P}}[\diag(\tilde{\xi}\tilde{\xi}^T)] =\diag(\mu\mu^T+ \sigma\sigma^T)
\end{array}
\right\},
\end{array}
\end{equation}
i.e., $v^*$ is the best possible upper bound on the expectation of
$\tilde{\xi}^TA\tilde{\xi}$ given only the means and variances in
$\tilde{\xi}$ with support contained in the unit hypercube.

The next lemma provides necessary and sufficient conditions for the
feasibility of \eqref{momb}.
\begin{lemma} \label{momfeas}
The moment problem \eqref{momb} is feasible if and only if $0 \leq
\sigma_i \leq \sqrt{\mu_i(1-\mu_i)}$for all $i = 1,\ldots,n$.
\end{lemma}
\begin{proof}
Necessity of $\sigma_i \geq 0$ arises from
$\mathbb{E}_{\mathbb{P}}[(\tilde{\xi}_i-\mathbb{E}_{\mathbb{P}}[\tilde{\xi}_i])^2]
\geq 0$. Necessity of $\sigma_i \leq \sqrt{\mu_i(1-\mu_i)}$ arises from
$\mathbb{E}_{\mathbb{P}}[\tilde{\xi}_i(1-\tilde{\xi}_i)] \geq 0$ since
$\xi_i \in [0,1]$. Sufficiency follows from considering the two point
random variable $\tilde{\xi}_i$:
\begin{equation*}
 \tilde{\xi}_i =
    \begin{cases}
      \mu_i - \sigma_i \sqrt{\frac{p_i}{1-p_i}} & \text{w.p.\ } 1-p_i,\\
        \mu_i +\sigma_i \sqrt{\frac{1-p_i}{p_i}} & \text{w.p.\ } p_i
    \end{cases}
    \end{equation*}
    where $p_i \in [\sigma_i^2/((1-\mu_i)^2+\sigma_i^2),\mu_i^2/(\mu_i^2+\sigma_i^2)]$ to ensure that the support of the random variable is in $[0,1]$ with $\mathbb{E}_{\mathbb{P}}[\tilde{\xi}_i] = \mu_i$ and $\mathbb{E}_{\mathbb{P}}[\tilde{\xi}_i^2] = \mu_i^2+\sigma_i^2$. The random vector $\tilde{\xi}$ can be constructed from the marginal distributions using independence.
\end{proof}

We next show that under the assumption that $-A$ is submodular, the
covariance-bounding moment problem admits a tractable SDP upper bound,
and this bound is exact in dimensions $n\le 3$.

\begin{proposition} \label{prop:tight}
Suppose:
\begin{enumerate}
    \item[(a)] $0 \leq \sigma_i \leq \sqrt{\mu_i(1-\mu_i)}$ for all $i = 1,\ldots,n$;
    \item[(b)] the matrix $-A$ is submodular.
    \end{enumerate}
Then the semidefinite program
      \begin{equation} \label{primal1momsdp3}
\begin{array}{rllll}
v^*_{\rm SDP} = \displaystyle \max & \displaystyle A \bullet \Sigma & \\
\st
& \displaystyle \diag(\Sigma) = \diag(\mu\mu^T+\sigma\sigma^T)\\
& \displaystyle \Sigma \leq e \mu^T\\
& \displaystyle \begin{pmatrix} 1 & \mu^T\\ \mu& \Sigma \end{pmatrix} \succeq 0.  &
\end{array}
\end{equation}
where the decision variable is $\Sigma \in \mathbb{S}^{n}$, provides an upper bound $v^* \le v^*_{\rm SDP}$ on the moment problem \eqref{momb}, with equality for $n \le 3$.
\end{proposition}

In \eqref{primal1momsdp3}, the diagonal constraint fixes the second
moments, the componentwise inequality enforces $\Sigma_{ij} \leq
\mu_j$, and the block positive-semidefinite constraint is the usual
moment-matrix consistency condition.

\begin{proof}
We proceed in four steps: dual reformulation of \eqref{momb},
replacement of the QSMB constraint by its SDP relaxation, dualization
of the inner SDP, and recovery of the primal semidefinite program
\eqref{primal1momsdp3}.

\textbf{Step 1 (dual reformulation).}
Assumption (a) and Lemma \ref{momfeas} imply (\ref{momb}) is feasible. Its dual is
\begin{equation*}
\begin{array}{rllll}
v_d^*  = \displaystyle \inf & \displaystyle y_0 + \mu^T y + \diag(\mu\mu^T+\sigma\sigma^T)^T z & \\
\st
& \displaystyle y_0 + \xi^T y + \xi^T \Diag({z}) \xi \geq \xi^T A \xi& \forall \ {\xi} \in [0,1]^n,
\end{array}
\end{equation*}
where the decision variables are $y_0 \in \mathbb{R}, y \in
\mathbb{R}^{n}$, and $z \in \mathbb{R}^{n}$.  Then strong duality
holds with $v^* = v_d^*$ because the dual is strictly feasible. For
example, we can set $y = z = 0$ and then take $y_0$ larger than the
maximum value of $\xi^T A \xi$ over $\xi \in [0,1]^n$. We next
reformulate the dual as
\begin{equation*}
\begin{array}{rllll}
v^*  = \displaystyle \inf & \displaystyle y_0 + \mu^T y + \diag(\mu\mu^T+\sigma\sigma^T)^T z & \\
\st
& \displaystyle v(z) := \min\left\{y_0 +  y^T \xi + \xi^T (\Diag({z})-A) \xi\ : \  {\xi} \in [0,1]^n\right\} \geq 0.&
\end{array}
\end{equation*}

\textbf{Step 2 (SDP relaxation of QSMB).}
Under assumption (b), $\Diag(z)-A$ is submodular. Replacing the QSMB
constraint with the SDP relaxation from Theorem \ref{the:n=3} gives the
following optimization:
\begin{equation*}
\begin{array}{rllll}
\rho^* = \displaystyle \inf &  \displaystyle y_0 + \mu^T y + \diag(\mu\mu^T+\sigma\sigma^T)^T z & \\
\st
& \displaystyle \rho(z) := \min\left\{y_0 + y^T \xi +
(\Diag({z})-A)\bullet \Xi\ : \  \Xi \leq \xi e^T, \begin{pmatrix}1
&\xi^T\\ \xi & \Xi \end{pmatrix} \succeq 0\right\} \geq 0.&
\end{array}
\end{equation*}
Note that $v(z) \ge \rho(z)$ with equality for $n \le 3$ by Theorem
\ref{the:n=3}, which means the SDP constraint is more restrictive than
the original QSMB constraint. Hence, $v^* \le \rho^*$ with equality
for $n \le 3$.

\textbf{Step 3 (dualization of the inner SDP).} Taking the dual of the minimization in the constraint,
where both the primal and dual semidefinite programs are strictly
feasible, and following a similar argument as in the proof of
Proposition \ref{pr:dro}, we get
\begin{equation*}
\begin{array}{rllll}
\rho^* = \displaystyle \inf &  \displaystyle \lambda + \mu^T y + \diag(\mu\mu^T+\sigma\sigma^T)^T z & \\
\st
& \displaystyle    Z \geq 0\\
&\displaystyle \begin{pmatrix} \textcolor{black}{\lambda} & ({y-Z^T e})^T/{2}\\ ({y-Z^Te})/{2} & \Diag({z})+(Z+Z^T)/2 \end{pmatrix} \succeq \begin{pmatrix} 0 & 0^T \\ 0 & A \end{pmatrix}.&
\end{array}
\end{equation*}
\textcolor{black}{Here, $\lambda$ is an auxiliary scalar variable that plays the
role of $y_0$. It arises as the top-left entry in the dual of the inner
SDP, exactly as $z_0$ does in the proof of Proposition \ref{pr:dro},
and the constraint $\rho(z) \ge 0$ forces $\lambda = y_0$ at optimality.
In particular, the block positive-semidefinite constraint implies
$\lambda \ge 0$, consistent with $y_0 \ge 0$.}

\textbf{Step 4 (recover the primal SDP).}
The primal semidefinite formulation is then given by
\eqref{primal1momsdp3}. In particular, $v_{\rm SDP}^* = \rho^* \ge
v^*$ with equality for $n \le 3$.
\end{proof}

\subsection{Closed-form solution in dimension 2}

We solve the moment problem (\ref{momb}) in closed form when $n=2$ and
\begin{equation} \label{equ:Adim2}
A = \frac12 \begin{pmatrix} 0 & 1 \\ 1 & 0 \end{pmatrix}
\end{equation}
by explicitly constructing the extremal distribution that attains the upper bound.

\begin{proposition} \label{n=2}
Consider (\ref{momb}) for $n = 2$ and $A$ given by (\ref{equ:Adim2}).
The upper bound is given by
\begin{equation}\label{cauchy1}
\begin{array}{rllll}
  \mathbb{E}_{\mathbb{P}}[\tilde{\xi}_1\tilde{\xi}_2] \leq v^* = v_{\rm SDP}^* = \min\left(\mu_1,\mu_2,\mu_1\mu_2+\sigma_1\sigma_2\right),
\end{array}
\end{equation}
and this bound is best possible for all feasible parameter values
$\mu_i := \mathbb{E}_{\mathbb{P}}[\tilde{\xi}_i]$ and $\sigma_i:=
\sqrt{\mbox{Var}[\tilde{\xi}_i]}$ for $i = 1,2$.
\end{proposition}

\begin{proof}
For $n = 2$, the semidefinite program \eqref{primal1momsdp3} reduces to
\begin{equation*}
\begin{array}{rllll}
v^* = v^*_{\rm SDP} =
\displaystyle \max \left\{\Sigma_{12} \ : \ \Sigma_{12} \leq \mu_1, \ \Sigma_{12} \leq \mu_2, \
\begin{pmatrix} 1 & \mu_1 & \mu_2\\ \mu_1 & \mu_1^2+\sigma_1^2 & \Sigma_{12} \\
\mu_2 & \Sigma_{12} & \mu_2^2+\sigma_2^2
\end{pmatrix} \succeq 0\right\}.
\end{array}
\end{equation*}
Using the Schur complement theorem, we rewrite this as
\begin{equation*}
\begin{array}{rllll}
v^* = \displaystyle \max \left\{\Sigma_{12} \ : \ \Sigma_{12} \leq \mu_1, \ \Sigma_{12} \leq \mu_2, \ \begin{pmatrix} \sigma_1^2 & \Sigma_{12}-\mu_1\mu_2 \\
\Sigma_{12}-\mu_1\mu_2& \sigma_2^2
\end{pmatrix} \succeq 0 \right\} \\
= \displaystyle \max \left\{\Sigma_{12} \ : \ \Sigma_{12} \leq \mu_1, \
 \Sigma_{12} \leq \mu_2, \ \mu_1\mu_2-\sigma_1\sigma_2 \leq \Sigma_{12} \leq \mu_1\mu_2+\sigma_1\sigma_2 \right\}.
\end{array}
\end{equation*}
It follows that $v^*  = \min \{ \mu_1,\mu_2,\mu_1\mu_2+\sigma_1\sigma_2 \}$.

We now construct the extremal joint distribution using the two-point
marginal distributions from Lemma \ref{momfeas} to show attainment of
the bound. Define $\underline{p}_i =
\sigma_i^2/((1-\mu_i)^2+\sigma_i^2)$ and $\overline{p}_i =
\mu_i^2/(\mu_i^2+\sigma_i^2)$ for $i = 1,2$, and note that
$\underline{p}_i \leq \overline{p}_i$ for $i = 1,2$.

\textbf{Case 1 (overlap).} Suppose the intervals
$[\underline{p}_1,\overline{p}_1]$ and
$[\underline{p}_2,\overline{p}_2]$ overlap, and let $p$ be any value in
the overlapping region. Consider the two-point joint distribution with
perfect positive dependence:
\begin{equation*}
 (\tilde{\xi}_1,\tilde{\xi}_2) =
    \begin{cases}
      \left( \mu_1 - \sigma_1 \sqrt{\frac{p}{1-p}},\mu_2 - \sigma_2 \sqrt{\frac{p}{1-p}} \right) & \text{w.p. } \ 1-p,\\
        \left( \mu_1 +\sigma_1 \sqrt{\frac{1-p}{p}},\mu_2 +\sigma_2 \sqrt{\frac{1-p}{p}} \right) & \text{w.p. } \ p.
    \end{cases}
    \end{equation*}
In this case,
    \begin{equation*}
\begin{array}{rllll}
  \mathbb{E}_{\mathbb{P}^*}[\tilde{\xi}_1\tilde{\xi}_2] & =
  & (1-p) \left( \mu_1 - \sigma_1 \sqrt{\frac{p}{1-p}} \right) \left(\mu_2 - \sigma_2 \sqrt{\frac{p}{1-p}} \right) + p \left( \mu_1 +\sigma_1 \sqrt{\frac{1-p}{p}} \right) \left(\mu_2 +\sigma_2 \sqrt{\frac{1-p}{p}} \right)\\
  &= &\mu_1\mu_2 + \sigma_1\sigma_2.
\end{array}
\end{equation*}
To verify that this bound is attained, note that the overlap condition
implies both $\underline{p}_2 \leq \overline{p}_1$ and
$\underline{p}_1 \leq \overline{p}_2$, which are equivalent to
$\mu_1\mu_2 + \sigma_1\sigma_2 \leq \mu_1$ and $\mu_1\mu_2 +
\sigma_1\sigma_2 \leq \mu_2$, respectively. Hence $v^* = \min \{
\mu_1,\mu_2,\mu_1\mu_2+\sigma_1\sigma_2 \} =
\mu_1\mu_2+\sigma_1\sigma_2$, and the bound is attained.

\textbf{Case 2 (non-overlap).} Suppose the intervals
$[\underline{p}_1,\overline{p}_1]$ and
$[\underline{p}_2,\overline{p}_2]$ do not overlap and, without loss of
generality, suppose $\overline{p}_1 \leq \underline{p}_2$. Consider the
three-point joint distribution with perfect positive dependence:
\begin{equation*}
 (\tilde{\xi}_1,\tilde{\xi}_2) =
    \begin{cases}
      (0,(\mu_2-\mu_2^2-\sigma_2^2)/(1-\mu_2)) & \text{w.p. } \ 1-\underline{p}_2,\\
        (0,1) & \text{w.p. } \ \underline{p}_2-\overline{p}_1,\\
          (\mu_1+\sigma_1^2/\mu_1,1) & \text{w.p. } \ \overline{p}_1.
    \end{cases}
    \end{equation*}
In this case, $\mathbb{E}_{\mathbb{P}^*}[\tilde{\xi}_1\tilde{\xi}_2] =
\mu_1$.  \textcolor{black}{Since $\overline{p}_1 \leq \underline{p}_2$, we have $\mu_1 \leq \mu_1\mu_2 + \sigma_1\sigma_2$. On the other
hand, the chain $\underline{p}_1 \leq \overline{p}_1 \leq
\underline{p}_2 \leq \overline{p}_2$ shows that $\underline{p}_1 \leq
\overline{p}_2$ holds, which is equivalent to $\mu_1\mu_2 +
\sigma_1\sigma_2 \leq \mu_2$. Combining these two inequalities yields
$\mu_1 \leq \mu_1\mu_2 + \sigma_1\sigma_2 \leq \mu_2$.} 
Hence $v^* =
\min \{ \mu_1,\mu_2,\mu_1\mu_2+\sigma_1\sigma_2 \} = \mu_1$, and the
bound is attained.
\end{proof}

\noindent The attainability of this upper bound for $n = 2$ has also
been recently shown in \cite{Cov} (see Theorem 2 therein), although their
proof technique differs from ours.

The $n=2$ result above shows when the bivariate upper bound is
attainable exactly. We now extend this perspective to general $n$ and
identify a condition under which all pairwise upper bounds are
simultaneously attainable.

\subsection{Closed-form solution in dimension \texorpdfstring{$n$}{n}}

We can also solve problem (\ref{momb}) in closed form for general $n$
under additional moment conditions when
\begin{equation} \label{equ:Adimn}
A = \frac12 \begin{pmatrix} 0 & 1  &\ldots & 1 \\ 1 & 0 &\ldots & 1\\
\vdots & \vdots & \ddots & \vdots\\
1 & 1 & \ldots & 0\end{pmatrix} = \frac12 (ee^T - I)
\end{equation}
by explicitly constructing the extremal distribution that attains the
upper bound.

\begin{proposition} \label{n-sumofub}
Consider (\ref{momb}) for general $n$ and $A$ given by (\ref{equ:Adimn}), where the marginal moments satisfy the condition
\begin{equation} \label{conditionfortight}
\max_{i=1,\ldots,n} \frac{\sigma_i^2}{(1-\mu_i)^2+\sigma_i^2} \leq \min_{i=1,\ldots,n} \frac{\mu_i^2}{\mu_i^2+\sigma_i^2}.
\end{equation}
In words, condition \eqref{conditionfortight} guarantees that all
intervals $[\underline{p}_i,\overline{p}_i]$ have a common overlap,
which allows one comonotone two-point construction to attain all
pairwise upper bounds at once.
Then the best possible upper bound is given by
\begin{equation}\label{cauchy-general}
\begin{array}{rllll}
 \displaystyle v^* = \sum_{i < j}\left(\mu_i\mu_j+\sigma_i\sigma_j\right).
\end{array}
\end{equation}
\end{proposition}
\begin{proof}
For general $n$, an upper bound on $v^*$ (not necessarily attained) is obtained by adding up the respective bivariate bounds from Proposition~\ref{n=2}:
\begin{equation}
\begin{array}{rllll}
  \displaystyle v^* \leq \sum_{(i,j): i < j} \min\left(\mu_i,\mu_j,\mu_i\mu_j+\sigma_i\sigma_j\right).
  \end{array}
  \end{equation}
We now construct the extremal joint distribution that attains this bound. Define $\underline{p}_i = \sigma_i^2/((1-\mu_i)^2+\sigma_i^2)$ and $\overline{p}_i = \mu_i^2/(\mu_i^2+\sigma_i^2)$, and note that $\underline{p}_i \leq \overline{p}_i$. Under condition (\ref{conditionfortight}), the intervals $[\underline{p}_i,\overline{p}_i]$ overlap for $i = 1,\ldots,n$. Let $p$ be any value in the overlapping region. Consider the following two-point marginal distribution for each random variable with perfect positive dependence:
\begin{equation*}
 \tilde{\xi}_i =
    \begin{cases}
      \mu_i - \sigma_i\sqrt{\frac{p}{1-p}} & \text{w.p. } \ 1-p,\\
     \mu_i+\sigma_i \sqrt{\frac{1-p}{p}} & \text{w.p. } \ p.
    \end{cases}
    \end{equation*}
In this case for all $(i,j)$ with $i < j$, we have:
    \begin{equation*}
\begin{array}{rllll}
  \mathbb{E}_{\mathbb{P}^*}[\tilde{\xi}_i\tilde{\xi}_j] & =
  & (1-p) \left( \mu_i - \sigma_i\sqrt{\frac{p}{1-p}} \right) \left(\mu_j - \sigma_j \sqrt{\frac{p}{1-p}} \right) + p \left( \mu_i +\sigma_i \sqrt{\frac{1-p}{p}} \right) \left(\mu_j +\sigma_j \sqrt{\frac{1-p}{p}} \right)\\
  &= &\mu_i\mu_j + \sigma_i\sigma_j.
\end{array}
\end{equation*}
Since condition \eqref{conditionfortight} ensures that the intervals
$[\underline{p}_i,\overline{p}_i]$ overlap for all pairs $(i,j)$, the
Case~1 analysis from Proposition~\ref{n=2} gives $\min \{
\mu_i,\mu_j,\mu_i\mu_j+\sigma_i\sigma_j \} =
\mu_i\mu_j+\sigma_i\sigma_j$ for each pair. Hence
\begin{equation}
\begin{array}{rllll}
  \displaystyle v^* \leq \sum_{(i,j): i < j} \left(\mu_i\mu_j + \sigma_i\sigma_j\right),
  \end{array}
  \end{equation}
  and the constructed distribution attains this bound, establishing equality.
\end{proof}

We have the following immediate corollary of Proposition
\ref{n-sumofub} because the overlap condition
\eqref{conditionfortight} is satisfied when all $\mu_i$ are equal and
all $\sigma_i$ are equal.

\begin{corollary}
The bound in Proposition \ref{n-sumofub} is attained for all $n$
when $\mu_i = \mu_0$ and $\sigma_i = \sigma_0$ for some $\mu_0$ and
$\sigma_0$ and all $i = 1,\ldots,n$.
\end{corollary}

\section{Numerical Results} \label{sec:numerical}

In support of Sections \ref{sec:tight}, \ref{sec:dro}, and
\ref{sec:cov}, we next investigate the numerical behavior of our
 SDP relaxation for QSMB.
Specifically, we first \textcolor{black}{examine its behavior
in} higher dimensions and \textcolor{black}{then
present numerical studies based on three examples in multi-product
pricing, distributionally robust stratified sampling allocation, and
moment bounds on the Laplacian quadratic energy.}

\subsection[Empirical behavior of the relaxation when n >= 4]{\textcolor{black}{Empirical behavior of the relaxation when $n \ge 4$}}
\label{sec:empirical}

Theorem \ref{the:n=3} establishes that the SDP relaxation
(\ref{equ:sdp}) is tight for $n \le 3$ but generally inexact for $n \ge 4$. To investigate how
empirical tightness extends to higher dimensions, we conducted a 
numerical study using 1{,}000 randomly generated submodular instances
with dimensions~$n$ sampled uniformly from $\{4, 5, \ldots, 20\}$.

Each instance is constructed as follows. A symmetric matrix $Q \in
\mathbb{R}^{n \times n}$ is generated by drawing entries from the
standard normal distribution and symmetrizing, after which all
off-diagonal entries are replaced by $-|Q_{ij}|$ to enforce the
submodular structure $Q_{ij} \le 0$ for $i \neq j$. A cost vector $c
\in \mathbb{R}^n$ is drawn independently from the standard normal
distribution. Note that no restrictions are placed on the signs of $c$
or the diagonal of $Q$, so that these instances are not covered by the
existing polynomial-time results of \cite{Kim-Kojima_2003} (which
require $c \le 0$) or those for DR-submodular functions (which require
$\diag(Q) \le 0$).

For each instance, we solve the QP (\ref{equ:qpb}) to global
optimality using Gurobi \citep{gurobi} and then solve the SDP
relaxation (\ref{equ:sdp}) using Mosek \citep{mosek}, both with
stringent solver tolerances ($10^{-8}$ for the QP and $10^{-10}$ for
the SDP). We measure the quality of the relaxation via the relative
gap
\[
\text{relative gap} := \frac{p^* - d^*}{\max\!\big\{1, |p^* + d^*|/2\big\}},
\]
where $p^*$ is the optimal value of the QP and $d^*$ is the dual
objective value of the SDP. We use the dual objective because it
provides a rigorous lower bound via weak duality. By weak duality,
$p^* \ge d^*$, so the relative gap is nonnegative, and a gap of zero
indicates tightness.

Figure \ref{fig:rel_gap} displays the distribution of relative gaps
across all 1{,}000 instances using a kernel density estimate with
gap values plotted on a logarithmic (base-10) scale. The gaps range from approximately $10^{-14}$ to
$10^{-9}$, with the density peaking near $10^{-10}$. These values are
well within solver precision, indicating that the SDP relaxation
(\ref{equ:sdp}) is tight to machine accuracy for every instance
tested. In particular, every relative gap is several orders of
magnitude below $10^{-4}$, a conservative threshold for declaring
numerical tightness.

\begin{figure}[!htbp]
\begin{center}
\includegraphics[scale=0.8]{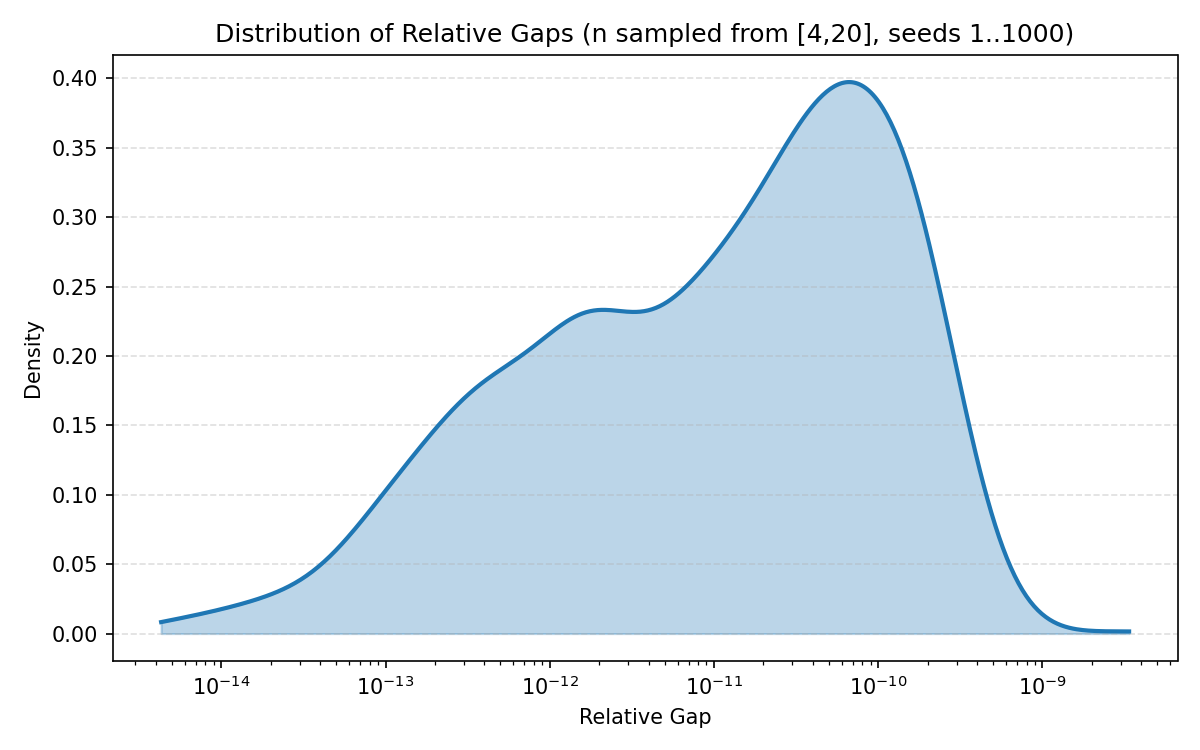}
\caption{Distribution of relative gaps between the QP optimal value
  and the SDP dual bound over 1{,}000 random submodular instances with
  $n$ sampled uniformly from $\{4, \ldots, 20\}$. All gaps are within
  solver precision. \textcolor{black}{Example \ref{exa:counter} shows that
  tightness nevertheless fails for some submodular instances.}}
\label{fig:rel_gap}
\end{center}
\end{figure}

\textcolor{black}{While these results suggest tightness occurs often for
every $n$, Example \ref{exa:counter} of course shows that tightness is
not a generic property. The experiment is therefore best understood as
evidence that violations of tightness are rare.}

\textcolor{black}{A further search at each fixed dimension $n \in
\{4,5,6,7\}$, comprising $34{,}000$ instances generated by the
same random submodular construction, likewise produced no relative
gap exceeding the $10^{-4}$ threshold. These runs also reveal a
structural regularity. In all but one of the $34{,}000$ instances the
optimal SDP solution satisfied $\rank(Y(x,X)) = 1$, and no s-lower
inequality was active at any of them. Any rank-one solution lies in
$\Cc\Pc(\Jc)$, so tightness follows immediately from Lemma \ref{lem:cp},
without recourse to the more delicate rank arguments of Propositions
\ref{pro:rank2}--\ref{pro:at_least_one_strictly_lower}. In particular,
none of the higher-rank configurations identified in Remark \ref{rem:n4}
arose in any of these randomly generated instances even though, by
Example \ref{exa:counter}, such configurations do occur and can produce
a gap. That configuration is therefore highly non-generic, which
explains how the relaxation can fail in principle while being tight on
every instance sampled here.}

\subsection{Pricing with $n = 3$ products}
\label{sec:pricing}
\textcolor{black}{We consider a pricing problem faced by a mobile data
provider, who sells three prepaid data bundles of 10 GB, 50 GB and 100
GB. The linear demand model is given by $d(p)=a-Bp$, where
\begin{equation*}
a=
\begin{pmatrix}
208\\
220\\
342
\end{pmatrix},
\quad\quad\quad
B=
\begin{pmatrix}
25 & -20 & -1\\
-3 & 4 & -\tfrac14\\
-1 & -\tfrac14 & 8
\end{pmatrix}.
\end{equation*}
The price variable $p$ is measured in dollars, and demand $d(p)$
is measured in thousands of bundles. The marginal costs in dollars,
baseline prices in dollars, baseline demands in thousands per month,
and admissible price ranges are provided in Table \ref{tab:P1}.
Over the admissible price ranges, the demands are nonnegative.
\begin{table}[tb]
\centering
\textcolor{black}{
\begin{tabular}{lrrrr}
\hline
Bundle & Marginal cost & Baseline price & Baseline demand & Price range \\ \hline
10 GB  & \$3  & \$8  & 400 & $[\$6,\$10]$ \\
50 GB & \$8  & \$18 & 180 & $[\$14,\$22]$ \\
100 GB & \$14 & \$32 & 98.5 & $[\$24,\$36]$ \\ \hline
\end{tabular}
\caption{Data for the product pricing problem where demand is measured in thousands.}
\label{tab:P1}
}
\end{table}
The matrix $(B+B^T)/2$ is indefinite, and so the pricing problem is
not directly solvable as a convex quadratic program. The baseline
profit, i.e., before optimization, equals \$5.573 million per month.
Since $n = 3$, Theorem \ref{the:n=3} guarantees that the SDP relaxation
(\ref{equ:sdp}) is tight, and solving it returns the exact optimum of
the pricing problem:
\begin{equation*}
p^*=
\begin{pmatrix}
10\\
22\\
30
\end{pmatrix},
\quad\quad\quad
d(p^*)=
\begin{pmatrix}
428\\
169.5\\
117.5
\end{pmatrix}.
\end{equation*}
The corresponding optimal profit equals \$7.249 million per month, an increase of
approximately $30.07\%$ over the baseline profit. The product-level
profit calculation is shown in Table \ref{tab:optimal-profit}, where the
optimal margin is the optimal price less the marginal cost, and the profit is
the margin times the optimal demand, expressed in millions of dollars per
month.
\begin{table}[tb]
\centering
\textcolor{black}{
\begin{tabular}{lrrrr}
\hline
Bundle & Optimal price & Optimal demand & Optimal margin & Profit (in millions) \\
\hline
10 GB  & 10 & 428 & \$7  & \$2.996 \\
50 GB  & 22 & 169.5 & \$14 & \$2.373 \\
100 GB & 30 & 117.5 & \$16 & \$1.880 \\ \hline
\end{tabular}
\caption{Optimal prices, demands, margins and product-level profits.}
\label{tab:optimal-profit}
}
\end{table}
In the optimal solution, the prices of the small and medium bundles are raised to their upper bounds, while the price of the largest bundle is lowered \$2 from its baseline.}

\subsection{Distributionally robust stratified sampling allocation} \label{sec:allocation}

\textcolor{black}{Allocation problems in stratified sampling are solvable
using optimization methods; see \cite{neyman,srikantan}. In this
subsection, we focus on a distributionally robust variant. Consider a
population of individuals where each individual is associated with a
random Bernoulli outcome variable and the probability of the Bernoulli
outcome taking a value of $1$ is itself random. Such random variables
are typically modeled using the Beta-Bernoulli distribution, but we
consider a distributionally robust approach instead of requiring the
probability to follow a Beta distribution.}

\textcolor{black}{Assume the population is divided into $n$ strata
indexed by $i$. Let $\tilde\xi_i\in[0,1]$ denote the probability of
the Bernoulli outcome taking a value of 1 in stratum $i$, which is
itself random. Conditional on $\tilde\xi_i$, the Bernoulli outcome for
individual $j$ from stratum $i$ is
\begin{equation*}
  \tilde{b}_{ij}|\tilde\xi_i
  \sim \operatorname{Bernoulli}(\tilde\xi_i).
\end{equation*}
When the outcomes are conditionally independent within each stratum and
across strata, the sample-proportion estimator is given by
\begin{equation*}
  \tilde{b}_i=\frac{1}{m_i}\sum_{j=1}^{m_i}\tilde{b}_{ij},
\end{equation*}
where $m_i$ is the number of samples allocated to stratum $i$.
The conditional mean and variance of the
estimator for stratum $i$ are given by
\begin{equation*}
  \mathbb{E}[\tilde{b}_i |\tilde\xi_i] = \tilde{\xi}_i, 
  \quad\quad\quad
  \mbox{Var}[\tilde{b}_i |\tilde\xi_i] = \frac{\tilde{\xi}_i(1-\tilde{\xi}_i)}{m_i}.
\end{equation*}
}

\textcolor{black}{Suppose the decision maker has a sampling budget $B$,
and suppose $c_i>0$ is the unit sampling cost for stratum $i$. When the
joint distribution of the probability vector $\tilde\xi$ is known and
given by $\mathbb{P}$, one can find an optimal allocation which
minimizes the ex-ante maximum variance by solving
\begin{equation}
  \min
  \left\{
    \mathbb{E}\left[
      \max_{i\in[n]}
      \frac{\tilde{\xi}_i(1-\tilde\xi_i)}{m_i}
    \right]
    :\; c^Tm\le B, m \geq 0 
  \right\}.
  \label{eq:known-distribution-allocation-1}
\end{equation}
See \cite{krause} for similar models that minimize the maximum
posterior variance in Gaussian process regression. Further, we consider
a distributionally robust version of this problem where the joint
distribution of the probabilities is unknown and the moment ambiguity
set ${\cal P}$ is defined as in
\eqref{eq:prob}:
\begin{equation}
  \min
  \left\{
    \sup_{\mathbb{P} \in {\cal P}}\mathbb{E}_{\mathbb{P}}\left[
      \max_{i\in[n]}
      \frac{\tilde{\xi}_i(1-\tilde\xi_i)}{m_i}
    \right]
    :\; c^Tm\le B, m \geq 0 
  \right\}.
  \label{eq:dro-distribution-allocation}
\end{equation}
For comparison, we also solve the problem with the ambiguity set ${\cal
Q}$ defined in \eqref{eq:del}.}

{\color{black} Since $\Xi = [0,1]^n$, the ambiguity sets ${\cal P}$ and
${\cal Q}$ are relevant in this context. The objective function inside
the expectation is given by the maximum of $n$ random quadratic terms
where the corresponding $A_i$ matrices are defined by diagonal
matrices with negative entries along the diagonals. This satisfies
the sufficient conditions discussed in Section \ref{sec:dro} for both
ambiguity sets. It is also straightforward to see that the formulation
is convex in $m$ by a simple transformation where we replace $1/m_i$ by
$m_i$.}

\textcolor{black}{For a fixed allocation $m$, the inner worst-case
expectation is precisely a moment problem of the form studied in
Section \ref{sec:dro}, and we evaluate it by solving the associated
semidefinite program. Since $n = 3$ and the objective is submodular,
Theorem \ref{the:n=3} guarantees that this semidefinite program
returns the exact worst-case value rather than a bound. The outer
minimization over $m$ is then handled by enumeration. Although
\eqref{eq:known-distribution-allocation-1} is convex in $m$ after the
substitution just described, sample sizes must in practice be integers,
and for the instance below there are only $406$ integer allocations to
consider, so we simply solve the inner semidefinite program at each of
them and select the best.}

\textcolor{black}{For a numerical example, consider $n = 3$ strata with
equal sampling costs $c = (1, 1, 1)^T$ and budget $B=30$. The $\mu$
vector and the $\Sigma$ matrix are given by
\begin{equation*}
   \mu = \begin{pmatrix}
    0.5\\
    0.5\\
    0.5
  \end{pmatrix},
  \quad\quad\quad
  \Sigma=
  \begin{pmatrix}
    0.49&0.25&0.25\\
    0.25&0.29&0.25\\
    0.25&0.25&0.26
  \end{pmatrix}.
\end{equation*}
We enumerate over all possible positive sampling integer allocations
with $m_1+m_2+m_3=30$, which yields
\begin{equation*}
m_{{\cal P}}^{*} = \begin{pmatrix}
    4\\
    13\\
    13
  \end{pmatrix},
  \quad\quad\quad
  m_{{\cal Q}}^{*} =  \begin{pmatrix}
    10\\
    10\\
    10
  \end{pmatrix}.
  \end{equation*}
Table~\ref{tab:ber} reports the worst-case objective value for each
ambiguity set at each of these two allocations: a row indexes the ambiguity
set over which the worst case is taken, and a column indexes the allocation at
which it is evaluated. The diagonal entries, shown in bold, are the optimal
values. Each off-diagonal entry is what one model delivers when the other
model's allocation is used.
\begin{table}[tb]
\centering
\textcolor{black}{
\begin{tabular}{lrr}
Ambiguity set
 & $m_{{\cal P}}^{*}$
 & $m_{{\cal Q}}^{*}$
\\
\hline
 ${\cal P}$
 & $\mathbf{0.0210}$ & $0.0250$\\
 ${\cal Q}$ 
 & $0.0625$ & $\mathbf{0.0250}$ \\
\hline
\end{tabular}
\caption{Objective functions for a given ambiguity set and given sampling allocation (optimal solutions are highlighted in bold).}
\label{tab:ber}
}
\end{table}
}

{\color{black}
The two allocations differ because the two ambiguity sets encode
different information. Under ${\cal P}$, the diagonal second moments are
specified exactly, so the expected conditional variance in each stratum
is known exactly as well:
\begin{equation*}
\mathbb{E}[\tilde\xi_i(1-\tilde\xi_i)] = \mu_i - \Sigma_{ii}
= 0.01,\ 0.21,\ 0.24 \quad \mbox{for } i = 1,2,3 .
\end{equation*}
Stratum 1 is an extreme case. Its probability $\tilde\xi_1$ varies a great
deal across the population, with $\mbox{Var}[\tilde\xi_1] = 0.24$, and yet
conditional on $\tilde\xi_1$ the Bernoulli outcome is nearly deterministic,
so very few samples are needed there. The allocation $m^*_{\cal P} =
(4,13,13)$ reflects precisely this, moving samples away from stratum 1 and
toward strata 2 and 3.

Under ${\cal Q}$, by contrast, only the positive semidefinite bound
$\mathbb{E}_{\mathbb{P}}[\tilde\xi\tilde\xi^T] \preceq \Sigma$ is imposed,
and the deterministic distribution placing all of its mass at $\tilde\xi =
(0.5,0.5,0.5)$ is feasible, because $\Sigma - 0.25\,ee^T \succeq 0$. Under
that distribution every stratum attains the largest conditional variance
possible, $\tilde\xi_i(1-\tilde\xi_i) = 1/4$, so all strata look equally
noisy and the equal allocation $m^*_{\cal Q} = (10,10,10)$ is optimal. The
same distribution accounts for the ${\cal Q}$ row of Table \ref{tab:ber}. At
any allocation it gives $\max_i 1/(4m_i)$, which is $0.0625$ at $m^*_{\cal
P}$ and $0.0250$ at $m^*_{\cal Q}$.

Two features of Table \ref{tab:ber} deserve comment. First, at $m^*_{\cal P}$
the two models differ by a factor of three, so the additional information
encoded in ${\cal P}$ translates into a materially better allocation. Second,
the two models agree exactly at $m^*_{\cal Q}$. This is not a coincidence:
$\max_i \tilde\xi_i(1-\tilde\xi_i) \le 1/4$ holds for every distribution
supported on $[0,1]^n$, so $1/(4 \cdot 10) = 0.0250$ is an upper bound at the
equal allocation for \emph{any} ambiguity set, and both ${\cal P}$ and
${\cal Q}$ contain a distribution attaining it. In other words, the equal
allocation is the one place where the extra information in ${\cal P}$ cannot
help. Taken together, the example shows that the new ambiguity set can deliver
strictly tighter bounds, and correspondingly better decisions, than the
existing one.

}
\subsection{Quadratic energy of Laplacian matrices}

Let $G = (V,E)$ be a simple undirected graph, where $V$ is the set of
vertices and $E$ is the set of edges. We focus on unweighted graphs,
but the results extend directly to nonnegative edge weights. The
Laplacian matrix $L$ associated with $G$ is a symmetric $|V| \times
|V|$ matrix with diagonal entries $L_{ii} = \text{deg}(i)$ for $i \in
V$ and off-diagonal entries $L_{ij} = -1$ when $i$ is adjacent to $j$
and $0$ otherwise. The Laplacian is submodular. It is also diagonally
dominant, and therefore positive semidefinite.

Associated with graph $G$ is the \textit{quadratic energy}
defined by
\[
E(\xi) :=\xi^T L\xi = \sum_{(i,j) \in E} (\xi_i-\xi_j)^2,
\]
where $\xi \in \mathbb{R}^{|V|}$. Minimization of $E(\xi)$ under appropriate constraints on $\xi$ has found applications in graph machine learning and spring and resistor networks, and there are fast algorithms available to solve these problems \citep{Spielman}.

\textcolor{black}{In spring and resistor networks, it is natural to
\emph{ground} some of the vertices, that is, to hold them at a fixed
reference value. In a resistor network, a grounded vertex is held at zero
potential, while in a spring network, it is a fixed anchor point rather
than a free mass. Accordingly, given a subset $S \subseteq
V$ of grounded vertices, let $G_S$ denote the graph $G$ augmented with a
single ground vertex $g$ joined to every $i \in S$, and pin $\xi_g = 0$.
Eliminating the pinned coordinate, the energy of $G_S$ becomes the
\emph{grounded quadratic energy}
\[
  E_S(\xi) := \xi^T L_S \xi
  = \sum_{(i,j) \in E} (\xi_i-\xi_j)^2 + \sum_{i \in S} \xi_i^2,
  \qquad L_S := L + \mbox{Diag}(\chi_S),
\]
where $\chi_S \in \{0,1\}^{|V|}$ is the indicator vector of $S$.
Grounding alters the diagonal only, so $L_S$ still has nonpositive
off-diagonal entries and is still diagonally dominant, and hence is
again both submodular and positive semidefinite. The ungrounded case is
recovered by taking $S = \emptyset$, for which $L_\emptyset = L$ and
$E_\emptyset = E$. Grounding does, however, change the null space: if
$G$ is connected and $S \neq \emptyset$, then $L_S$ is positive
definite, whereas $L$ is always singular along the constant vector $\xi
= te$. This distinction drives Example \ref{exa:subquantile} below.}

Because \textcolor{black}{$L_S$ is} both submodular and positive
semidefinite \textcolor{black}{for every $S \subseteq V$}, it
provides a natural setting in which to apply the results of Sections
\ref{sec:dro} and \ref{sec:cov}. We present two examples: the first
concerns distributionally robust subquantile bounds on the energy
\textcolor{black}{of a grounded graph} (illustrating
Section~\ref{sec:dro}), and the second concerns minimum
expected energy \textcolor{black}{of an ungrounded graph} under marginal
moment information (illustrating Section~\ref{sec:cov}). \textcolor{black}{Moment bounds on the Laplacian energy of graphs have also been studied in the signal processing community; see \cite{so}}.

 \begin{example} \label{exa:subquantile} Given a distribution $\mathbb{P}$ for $\tilde{\xi}$ and a parameter $\alpha \in [0,1)$, the expectation in the lower $(1-\alpha)$-tail distribution (or $(1-\alpha)$-subquantile) of the quadratic energy is given by the optimal value (see \cite{Rockafellar})
 \begin{equation} \label{equ:local}
\begin{array}{rlll}
\displaystyle \sup_{x}  \left(x+\frac{1}{1-\alpha} \cdot \mathbb{E}_{\mathbb{P}}\left[\min\left(0,E_S(\tilde{\xi})-x\right)\right]\right).
\end{array}
\end{equation}
Intuitively, the $(1-\alpha)$-subquantile captures the expected energy
in the lower tail of the distribution.
When $\alpha = 0$, the optimal value is the expected value $\mathbb{E}_{\mathbb{P}}[E_S(\tilde{\xi})]$ and when $\alpha \uparrow 1$, the optimal value converges to the minimum value of $E_S(\xi)$ over all $\xi$ in the support.

Now assume the distribution $\mathbb{P}$ lies in the ambiguity set
 \begin{equation*}
\begin{array}{lll}
   {\cal R} = \left\{
   \mathbb{P} \in {\cal P}([0,1]^n)  \ : \ \mathbb{E}_{\mathbb{P}}[\tilde{{\xi}}] = {\mu}, \ \mathbb{E}_{\mathbb{P}}[\tilde{{\xi}}\tilde{{\xi}}^T] = \Sigma
   \right\},
 \end{array}
 \end{equation*}
 where the support of $\tilde{{\xi}}$ is contained in $[0,1]^n$, the mean is fixed at $\mu$, and the second moment matrix is fixed at $\Sigma$. Then the worst-case $(1-\alpha)$-subquantile is given by the optimal value of
  \begin{equation} \label{equ:locala}
\begin{array}{rlll}
\displaystyle \sup_{x} \inf_{\mathbb{P} \in {\cal R}} \left(x+\frac{1}{1-\alpha} \cdot \mathbb{E}_{\mathbb{P}}\left[\min\left(0,E_S(\tilde{\xi})-x\right)\right]\right).
\end{array}
\end{equation}
 Solving this DRO problem with ambiguity set ${\cal R}$ is computationally intractable. However, we can use ${\cal P}$ defined in (\ref{eq:prob}) and ${\cal Q}$ defined in (\ref{eq:del}) as tractable approximations of ${\cal R}$. These yield independently computed lower bounds on the worst-case $(1-\alpha)$-subquantile.

\textcolor{black}{We compare the two bounds on the path graph $1 - 2 -
\cdots - 50$ with $n = 50$ vertices, grounded at its two endpoints,
that is, with $S = \{1,50\}$. In the spring and resistor network
interpretation, this anchors the two ends of the path.} We set $\mu$
and $\Sigma$ to the first two moments of the independent uniform random
vector on $[0,1]^n$. To apply the framework of Section \ref{sec:dro}, we
rewrite \eqref{equ:locala} equivalently as
\begin{equation*}
\begin{array}{rlll}
\displaystyle   -\inf_{x} \sup_{\mathbb{P} \in {\cal R}} \left(-x+\frac{1}{1-\alpha} \cdot \mathbb{E}_{\mathbb{P}}\left[\max\left(0,x-E_S(\tilde{\xi})\right)\right]\right)
\end{array}
\end{equation*}
and all conditions of Section \ref{sec:dro} apply, because the
\textcolor{black}{grounded Laplacian $L_S$} is both submodular and positive
semidefinite. Hence, we use
the SDPs in \eqref{primal1momsdp} and \eqref{primal2momsdp} for the
relaxed ambiguity sets ${\cal P}$ and ${\cal Q}$, respectively, to
obtain the corresponding lower bounds on \eqref{equ:locala}. The
results for different values of $\alpha$ are shown in Figure
\ref{fig:bounds}. For $\alpha \leq 0.2$, the bound from ${\cal P}$ is
stronger; for $\alpha > 0.2$, it is weaker. Thus, while both ambiguity
sets provide tractable relaxations, neither subsumes the other.

\begin{figure}
\begin{center}
\includegraphics[scale=0.8]{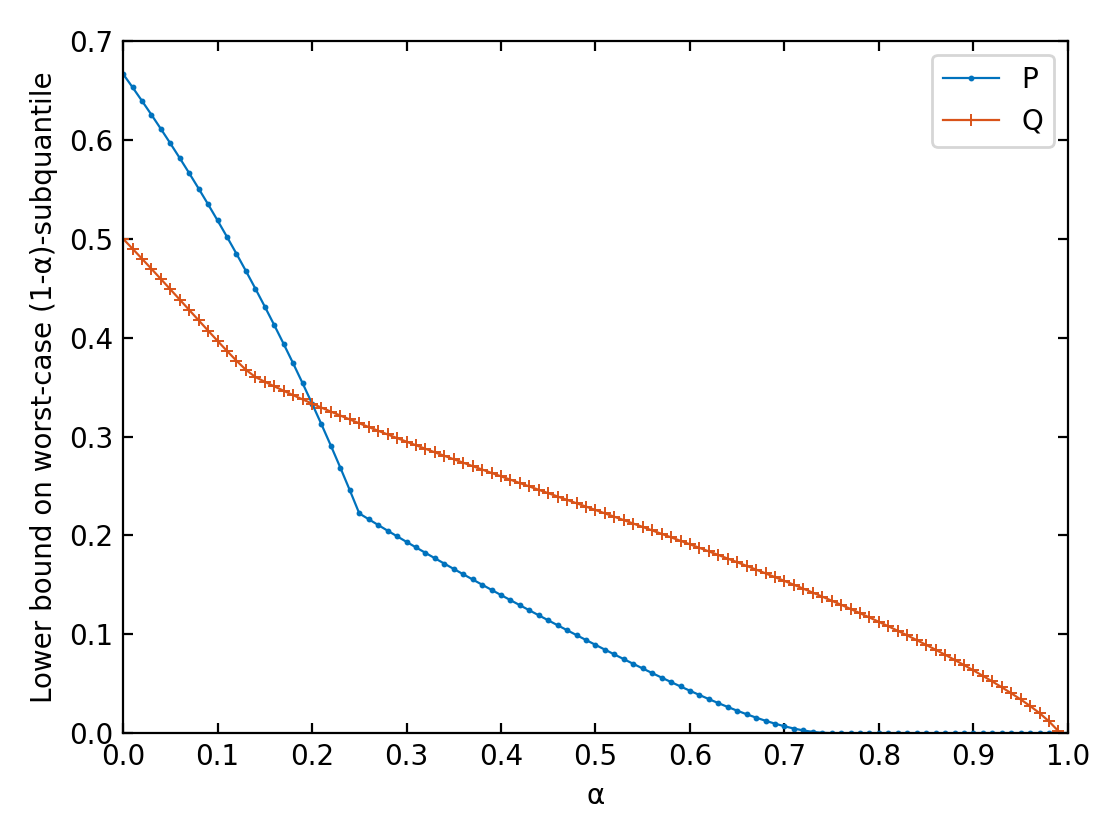}
\caption{\textcolor{black}{Lower bounds on the worst-case} $(1-\alpha)$-subquantile for ambiguity sets ${\cal Q}$ and ${\cal P}$.}
\label{fig:bounds}
\end{center}
\end{figure}

One could also combine ${\cal P}$ and ${\cal Q}$ to develop tighter bounds for ${\cal R}$. Such an approach using infimal convolution has been adopted in the DRO literature \citep{Droconvex}, but we leave a detailed analysis for future research.
\end{example}

Our next example illustrates the results of Section \ref{sec:cov} by
comparing the full SDP bound to the simpler sum of bivariate bounds,
showing how the gap between them varies with graph type and size.

 \begin{example} \label{exa:energy}
 Given the mean $\mu$ and standard deviation $\sigma$ of the random vector $\tilde\xi$, we compute the minimum expected energy $e^*$ over all distributions supported on the unit hypercube:
 \begin{equation*}
\begin{array}{rllll}
e^* := \displaystyle\min  \left\{\mathbb{E}_{\mathbb{P}}[E(\tilde{\xi})] \ :
\begin{array}{c}
\mathbb{P} \in {\cal P}([0,1]^n), \
\mathbb{E}_{\mathbb{P}}[\tilde{\xi}] = \mu, \mathbb{E}_{\mathbb{P}}[\diag(\tilde{\xi}\tilde{\xi}^T)] =\diag(\mu\mu^T+ \sigma\sigma^T)
\end{array}
\right\}.
\end{array}
\end{equation*}
\textcolor{black}{Here we take $S = \emptyset$, so that the energy is the ungrounded $E(\xi) = \xi^T L \xi$ and decomposes as a pure sum over edges.} The results of Section \ref{sec:cov} can be used to calculate $e^*$ via the problem \eqref{momb} with $A = -L$ where $L$ is the Laplacian of the graph. For comparison, since the Laplacian energy decomposes as a sum over edges, a lower bound $\underline{e}^*$ of $e^*$ can be obtained by independently bounding each edge term $(\tilde\xi_i - \tilde\xi_j)^2$ using the bivariate result of Proposition \ref{n=2} and summing.

In the following numerical experiments, we consider three types of graphs---the path graph, the star graph, and the complete graph, each with $n = 2, 10, 20,$ and $50$ vertices. We generated \textcolor{black}{1000} random instances for each type of graph and each value of $n$, where the mean vector $\mu$ and the standard deviation vector $\sigma$ are randomly generated to satisfy the feasibility conditions in Lemma \ref{momfeas}.

In Table \ref{energy}, we report the percentage gap between the $e^*$ and its lower bound $\underline{e}^*$, i.e., $(e^* - \underline{e}^*) / {e}^* \times 100\%$. \textcolor{black}{The results show that, for the star and complete graphs, the gap grows as the size of the graph grows, while for the path graph it remains roughly constant near $0.6\%$. Moreover, at every size the average gap is highest for the complete graph, followed by the star graph and then the path graph.}

 \begin{table}[!hbtp]
 \centering
\begin{tabular}{|c|c|c|c|}\hline
& Path & Star & Complete \\\hline
$n = 2$ & 0 (0) & 0 (0) &  0 (0)\\ \hline
$n = 10$ & \textcolor{black}{0.577 (0.872)} & \textcolor{black}{1.253 (2.007)} & \textcolor{black}{1.765 (1.052)}\\ \hline
$n = 20$ & \textcolor{black}{0.628 (0.603)} & \textcolor{black}{1.501 (1.892)} & \textcolor{black}{2.094 (0.746)}\\ \hline
$n = 50$ & \textcolor{black}{0.609 (0.345)} & \textcolor{black}{1.776 (2.046)} & \textcolor{black}{2.329 (0.455)}\\ \hline
\end{tabular}
\caption{Mean (standard deviation) of the percentage gap $\frac{(e^*-\underline{e}^*)}{{e}^*}\times 100\%$ computed over \textcolor{black}{1000} instances.}
\label{energy}
\end{table}
 \end{example}

\section{Additional Proofs} \label{sec:techlemma}

In this section, we present proofs of Propositions \ref{pro:rank2},
\ref{pro:diag}, and \ref{pro:at_least_one_strictly_lower}, which are
used to establish Theorem \ref{the:n=3}.

\subsection{Proof of Proposition \ref{pro:rank2}} 

To prove Proposition \ref{pro:rank2}, we need two lemmas. We start by
showing that a rank-2 feasible solution $Y(x,X)$, which satisfies an
additional (nonlinear) inequality condition, is completely positive
with respect to $\Jc$. 

\begin{lemma} \label{lem:rank2cp}
    Let $(x,X)$ be feasible for (\ref{equ:sdp}) with $r(x,X) = 2$ and
    $X \ge xx^T$. Then $Y(x,X) \in \Cc\Pc(\Jc)$.
\end{lemma}

\begin{proof}
Denote $Y := Y(x,X)$. It is well-known that $\rank(Y) = \rank(X -
xx^T) + 1$, which implies $\rank(X - xx^T) = 1$. Combining $X - xx^T
\succeq 0$ and $X - xx^T \ge 0$, we see $X - xx^T = ww^T$ where $w =
\sqrt{\diag(X) - x \circ x} \ge 0$. (Indeed, a rank-1 PSD matrix can be
written as $ss^T$; entrywise nonnegativity then implies all entries of $s$
have the same sign, and we take $w = |s|$.) Then
\[
Y = 
\begin{pmatrix} 1 & x^T \\ x & xx^T + ww^T \end{pmatrix} =
\begin{pmatrix} 1 & 0 \\ x & w \end{pmatrix}
\begin{pmatrix} 1 & 0 \\ x & w \end{pmatrix}^T.
\]
We wish to apply an orthogonal rotation of the form
\begin{equation} \label{equ:rotation}
\begin{pmatrix} \alpha & \beta \\ u & v \end{pmatrix} :=
\begin{pmatrix} 1 & 0 \\ x & w \end{pmatrix}
\begin{pmatrix} \cos\theta & \sin\theta \\ -\sin\theta & \cos\theta \end{pmatrix},
\quad\quad
\theta \in (0, \pi/2),
\quad\quad
\begin{pmatrix} \alpha \\ u \end{pmatrix},
\begin{pmatrix} \beta \\ v\end{pmatrix} \in \Jc.
\end{equation}
Note that $\alpha = \cos\theta > 0$ and $\beta = \sin\theta > 0$ for
$\theta \in (0, \pi/2)$. This rotation would establish that $Y \in
\Cc\Pc(\Jc)$ because this construction ensures
\[
Y = \begin{pmatrix} \alpha & \beta \\ u & v \end{pmatrix} \begin{pmatrix} \alpha & \beta \\ u & v \end{pmatrix}^T.
\]
It thus remains to show that system (\ref{equ:rotation}) is feasible by
demonstrating $0 \le u \le \alpha \, e$ and $0 \le v \le \beta
\, e$, which are equivalent to
\begin{align*}
    &0 \le \alpha x - \beta w \le \alpha e \quad \Leftrightarrow \quad 0 \le x - (\beta/\alpha) w \le e \\
    &0 \le \beta x + \alpha w \le \beta e \quad \Leftrightarrow \quad 0 \le x + (\alpha/\beta) w \le e.
\end{align*}

Because $0 \le x \le e$ and $\beta/\alpha > 0$, two of these four
vector inequalities necessarily hold; the remaining ones to verify are
\[
0 \le x - (\beta/\alpha) w
\quad\quad \text{and} \quad\quad
x + (\alpha/\beta) w \le e.
\]
Whenever $w_i = 0$, both vector inequalities hold.  Furthermore, when
$x_i \in \{0,1\}$, we must have $x_i^2 = X_{ii} = x_i$ and hence $w_i
= \sqrt{X_{ii} - x_i^2} = 0$. So we may assume $w > 0$ and $0 < x < e$
without loss of generality in which case the two inequalities read
\[
(e - x)^{-1} \circ w \le \frac{\beta}{\alpha} \le x \circ w^{-1},
\]
where inverses are componentwise and are well-defined because $0 < x < e$ and
$w > 0$, and where $\beta/\alpha = \tan\theta$. Define
\[
L := \max_i \, (1 - x_i)^{-1} w_i, \quad\quad U := \min_i \, x_i w_i^{-1}.
\]
Then it suffices to show $L \le U$, i.e.,
$(1 - x_i)^{-1} w_i \le x_j w_j^{-1}$ for all $(i,j)$, which is true because
\[
w_i w_j = X_{ij} - x_i x_j \le x_j - x_i x_j = (1 - x_i) x_j,
\]
where $X_{ij} \le x_j$ follows from symmetry and feasibility: $X_{ij} = X_{ji}
\le x_j$ since $X \le xe^T$. Hence $L \le U$, and we can choose
$\tan\theta \in [L,U]$, which yields a feasible $\theta \in (0,\pi/2)$.
\end{proof}

Our next lemma considers what happens when the nonlinear condition $X
\ge xx^T$ of the previous lemma is not satisfied. 

\begin{lemma} \label{lem:rank2notopt}
Let $(x,X)$ be feasible for (\ref{equ:sdp}) with $r(x,X) = 2$ and $X
\not \ge xx^T$. Then there exists another feasible $(\hat x, \hat X)$
with no worse objective value such that one of the following holds:
\begin{itemize}

\item[(i)] $r(\hat x, \hat X) = 3$ and $a(\hat x, \hat X) > a(x,X)$;

\item[(ii)] $r(\hat x, \hat X) = 2$ and $\hat X \ge \hat x \hat
x^T$.

\end{itemize}
\end{lemma}

\begin{proof}
Our proof strategy is to construct a line segment $\{ Y(\lambda) :
\lambda \in [0,1] \}$, which starts at $Y(0) := Y := Y(x,X)$ and ends
at another rank-2, positive semidefinite matrix $Y(1)$. In addition,
$Y(\lambda)$ is feasible for (\ref{equ:sdp}) for sufficiently small
$\lambda > 0$ and has non-increasing objective value along the
segment. Furthermore, $\rank(Y(\lambda)) = 3$ for all $\lambda \in
(0,1)$. Then we define $\lambda^*$ to be the supremum of $\lambda \in
[0,1]$ such that $Y(\lambda)$ is feasible and show that $Y(\lambda^*)$
satisfies either (i) or (ii).

To construct the line segment, we first identify the symmetric set of
indices where $X_{ij}<x_ix_j$, which is nonempty by assumption:
\[
\Sc := \{ (i,j) : X_{ij} < x_i x_j \} \ne \emptyset.
\]
Also,
\[
X_{ij} < x_i x_j \le \min\{x_i,x_j\} \quad \forall (i,j) \in \Sc,
\]
so the corresponding linear inequalities $X_{ij} \le \min\{x_i,x_j\}$
are inactive on $\Sc$.

As in the proof of Lemma \ref{lem:rank2cp}, since $\rank(Y) = 2$, we
have $\rank(X - xx^T) = 1$ and $X - xx^T \succeq 0$, so $X - xx^T =
ww^T$ for some vector $w$ with $|w| = \sqrt{\diag(X) - x \circ x}$.
Then
\[
Y = \begin{pmatrix}1 & x^T\\ x & xx^T+ww^T\end{pmatrix},
\]
and since $X \not\ge xx^T$, the vector $w$ must have mixed signs.

For $\lambda \in [0,1]$, define the desired line segment via the
convex combination
\[
Y(\lambda) := (1 - \lambda) Y + \lambda Y(1), \qquad \text{where} \qquad
Y(1) := \begin{pmatrix} 1 & x^T \\ x & xx^T + |w||w|^T \end{pmatrix}.
\]
These
relationships can be equivalently expressed via the definitions
\[
\Delta := |w||w|^T - ww^T,
\qquad
x(\lambda) := x,
\qquad
X(\lambda) := X + \lambda \Delta,
\qquad
Y(\lambda) := Y(x(\lambda),X(\lambda)).
\]
Note that $Y(1)$ is positive semidefinite and rank-2 and satisfies
$X(1) \ge x(1) x(1)^T$.  Note also that the $x(\lambda)$ is fixed
along the segment, while $X(\lambda)$ changes.  In particular, let $P
:= \{i : w_i > 0\}$ and $N := \{i : w_i < 0\}$. Since $X_{ij} = x_ix_j
+ w_iw_j$, we have $X_{ij} < x_ix_j$ if and only if $w_iw_j < 0$,
i.e., $(i,j) \in P \times N \cup N \times P$. Hence
$
\Sc = P \times N \cup N \times P,
$
and in particular no diagonal pair belongs to $\Sc$. Moreover,
\[
\Delta_{ij}>0 \text{ for } (i,j)\in \Sc,
\qquad
\Delta_{ij}=0 \text{ for } (i,j)\notin \Sc.
\]
This ensures that $X(\lambda)_{ij}$ is strictly increasing in
$\lambda$ for $(i,j)\in \Sc$, while all other entries of $X(\lambda)$
remain fixed.

We investigate the feasibility of $Y(\lambda)$. Because $Y(\lambda)$
is a convex combination of $Y$ and $Y(1)$, which are both positive
semidefinite, $Y(\lambda)$ is positive semidefinite for all $\lambda
\in [0,1]$. Furthermore, because $w$ has mixed signs, $w$ and $|w|$
are linearly independent, so $(1-\lambda)ww^T+\lambda |w||w|^T$ has
rank $2$ for every $\lambda\in(0,1)$.  By the standard Schur
complement rank identity, it follows that
\[
\rank(Y(\lambda)) = 1 + \rank\big((1-\lambda)ww^T+\lambda |w||w|^T\big)=3
\qquad \forall \lambda \in (0,1).
\]
With regards to the linear inequalities defining the feasible region,
the sign pattern of $\Delta$ discussed in the prior paragraph as well
as the fact that $X_{ij} < \min \{x_i, x_j\}$ for all $(i,j) \in \Sc$
ensures $X(\lambda) \le x(\lambda) e^T = xe^T$ for sufficiently small
$\lambda > 0$. We conclude that, for sufficiently small $\lambda>0$,
$Y(\lambda)$ is feasible.

Next we consider the change in objective value along the line
segment.  Since $x(\lambda) = x$ is constant, the change in objective
value is determined by
\[
Q\bullet X(\lambda)=Q\bullet X + \lambda\, Q\bullet\Delta,
\]
and therefore the sign pattern of $\Delta \ge 0$ discussed above,
which is supported only on off-diagonal indices in $\Sc$, implies $Q
\bullet \Delta \le 0$ by submodularity. So the objective value is
non-increasing along the segment.

Now define
\[
\lambda^*:=\sup\{\lambda\in[0,1]: X(\lambda)\le xe^T\},
\]
and set $(\hat x,\hat X):=(x,X(\lambda^*))$. If $\lambda^*<1$, then at
least one inequality indexed by $\Sc$ is active at $(\hat x,\hat X)$,
so $a(\hat x,\hat X)>a(x,X)$, and since $\lambda^*\in(0,1)$ we have
$r(\hat x,\hat X)=3$; this is case (i). If $\lambda^*=1$, then
$X(1) \ge xx^T$ and $\rank(Y(1)) = 2$ as noted above; this is case (ii).
\end{proof}

By combining the above two lemmas, we can now prove Proposition \ref{pro:rank2}.

\begin{proof}[Proof of Proposition \ref{pro:rank2}]
If $X \ge xx^T$, then Lemma \ref{lem:rank2cp} implies that $Y(x,X)$ is
completely positive with respect to $\Jc$, and hence by Lemma \ref{lem:cp}, the
relaxation is tight. On the other hand, if $X \not \ge xx^T$, Lemma
\ref{lem:rank2notopt} implies that there exists an alternative optimal solution
$(\hat x, \hat X)$ with either $r(\hat x, \hat X) = 2$ and $\hat X \ge
\hat x \hat x^T$ or $r(\hat x, \hat X) = 3$ and more active inequalities.
In the former case, Lemma \ref{lem:rank2cp} again implies that the relaxation is
tight.

In the latter case, consider the minimal face defined by $(\hat x,
\hat X)$, which necessarily excludes $(x,X)$ because there are more
active inequalities. Within that face, consider an extreme point with
rank at most 2. If such a point exists, repeat the same argument at
that new rank-2 point: either the SDP relaxation is tight, or we obtain
another rank-3 optimal point with more active inequalities.
Since the number of inequalities is finite, this process terminates,
yielding a proper subface of the optimal face in which all extreme
points have rank at least 3.
\end{proof}

\subsection{Proof of Proposition \ref{pro:diag}}

We start with a key lemma showing that $Y(x,X)$ is completely positive
with respect to $\Jc$ whenever all s-upper inequalities are active and
at most one diagonal inequality is inactive. We remark that the first
statement in the lemma does not assume $Y(x,X) \succeq 0$.

\begin{lemma} \label{lem:nearly_maximal_combined}
Let $(x,X)$ be sorted by $x$ with all s-upper inequalities active. If
all diagonal inequalities are active, then $Y(x,X) \in \Cc\Pc(\Jc)$.
The same conclusion holds if $(x,X)$ is feasible for (\ref{equ:sdp})
and all diagonal inequalities are active except for a single $X_{jj}
< x_j$.
\end{lemma}

\begin{proof}
Let $Y := Y(x,X)$. We have $X_{ij} = \min \{ x_i, x_j \}$ for all
index pairs except possibly for a single $X_{jj} < x_j$. We prove the
result by induction on $n$.

For the base case $n = 1$, if $X_{11} = x_1$, we verify that
\[
    Y = \begin{pmatrix} 1 & x_1 \\ x_1 & x_1 \end{pmatrix} = x_1
    \begin{pmatrix} 1 \\ 1 \end{pmatrix} \begin{pmatrix} 1 \\ 1
    \end{pmatrix}^T + (1 - x_1) \begin{pmatrix} 1 \\ 0 \end{pmatrix}
    \begin{pmatrix} 1 \\ 0 \end{pmatrix}^T \in \Cc\Pc(\Jc).
\]
If $X_{11} < x_1$ and $x_1^2 = X_{11}$, then $Y$ is rank-1 and hence
in $\Cc\Pc(\Jc)$. If $X_{11} < x_1$ and $x_1^2 < X_{11}$, then $Y
\succ 0$, and choosing $\eps > 0$ maximally such that $Y_\eps := Y -
\eps \, ee^T \succeq 0$ with $e := (1, 1)^T$ yields $Y_\eps = uu^T$
for some vector $u$. The linear inequalities in $\Lc$ are preserved as
$Y$ shifts to $Y_\eps$ because $ee^T$ satisfies all inequalities in
$\Lc$ with equality. Hence, $Y_\eps \in \Lc$, and thus $u \in \Jc$,
which in turn implies $Y_\eps = uu^T + \eps \, ee^T \in \Cc\Pc(\Jc)$.

Now consider $n > 1$. If $x_1 = 1$, then $x = e$ and $Y =
\tbinom{1}{e}\tbinom{1}{e}^T \in \Cc\Pc(\Jc)$. Suppose $x_1 < 1$ and
either all diagonal inequalities are active or $j \neq 1$, where
$j$ is the index of the inactive diagonal. In both
situations, $X_{11} = x_1$, and so the second column of $Y$ equals
$x_1 e$. Define $\tilde x_i := (x_i - x_1)/(1 - x_1)$ for $i = 2,
\ldots, n$. Then
\[
    Y = x_1 \, ee^T + (1 - x_1) \, \tilde Y,
\]
where $\tilde Y$ has zero second row and column. After deleting
them, $\tilde Y$ satisfies the conditions of the lemma with respect to
$\tilde x$ (note $0 \le \tilde x_2 \le \cdots \le \tilde x_n \le 1$):
when all diagonals are active, $\tilde Y$ again has all diagonals
active; when at most one diagonal is inactive, the inactive diagonal
persists in $\tilde Y$; and $\tilde Y
\succeq 0$ because $Y - x_1 \, ee^T$ is the Schur complement of
$X_{11} = x_1$ in $Y \succeq 0$. By the induction hypothesis, $\tilde
Y \in \Cc\Pc(\Jc)$, and since $ee^T \in \Cc\Pc(\Jc)$, the result
follows.

Finally, suppose $x_1 < 1$ and there exists an inactive diagonal
inequality at $j = 1$. Note $x_1 > 0$, since feasibility gives $X_{11}
\ge x_1^2$ and $X_{11} < x_1$ forces $x_1 > 0$. If $x_n = 1$, the
first and last rows of $Y$ coincide because $X_{nn} = x_n = 1$ and
$X_{in} = x_i$ for all $i$. Then the result follows by induction on
the reduced matrix. If $x_n < 1$, the difference of the first and last
columns of $Y$ is $(1 - x_n) \, v$ with $v := (1, 0, \ldots, 0)^T$, so
$v$ is in the range space of $Y$.  Choosing $\eps > 0$ maximally so
that $Y - \eps \, vv^T \succeq 0$ reduces the rank while changing only
the top-left entry, preserving all linear inequalities in $\Lc$. After
rescaling by the top-left entry, the resulting matrix retains the same
structure, so the process can be repeated until the rank reaches
$1$---giving $Y \in \Cc\Pc(\Jc)$ by construction---or $x_n = 1$ after
rescaling, which is handled above.
\end{proof}

We are now ready to prove Proposition \ref{pro:diag}.

\begin{proof}[Proof of Proposition \ref{pro:diag}]
Given optimal $(x,X)$ with $\diag(X) = x$, assume $(x,X)$ is sorted by
$x$ without loss of generality. Define $(\hat x, \hat X)$ by $\hat x
:= x$ and $\hat X_{ij} := \min \{ x_i, x_j \}$ for all $i,j$, so that
all s-upper and diagonal inequalities are active for $(\hat x, \hat
X)$.  Then $Y(\hat x, \hat X)$ is completely positive with respect to
$\Jc$ by Lemma \ref{lem:nearly_maximal_combined} and hence feasible for
(\ref{equ:sdp}).  Moreover, since $(\hat x, \hat X)$ differs from $(x,
X)$ only in the off-diagonal entries of $\hat X$ such that
$\text{offdiag}(X) \le \text{offdiag}(\hat X)$, the submodularity of
$Q$ implies $(\hat x, \hat X)$ is optimal for (\ref{equ:sdp}) as well.
This proves the proposition by Lemma \ref{lem:cp}. The second case
follows directly from Lemma \ref{lem:nearly_maximal_combined}.

\end{proof}

\subsection{Proof of Proposition \ref{pro:at_least_one_strictly_lower}}

The proof of Proposition \ref{pro:at_least_one_strictly_lower} relies
on the following lemma, which shows that if an s-lower inequality is
active, then the corresponding columns of $Y(x,X)$ are identical.

\begin{lemma} \label{lem:lower_active_duplicate}
Let $(x,X)$ be feasible for (\ref{equ:sdp}). After sorting by $x$,
suppose that, for $i < j$, the s-lower inequality $X_{ij} \le x_j$ is
active.  Then the columns of $Y(x,X)$ indexed by $i$ and $j$ are
identical.
\end{lemma}

\begin{proof}
Since $X_{ij} = x_j$ and $X_{ij} \le x_i \le x_j$, we have $x_i =
x_j = X_{ij}$. The $3 \times 3$ principal submatrix of $Y := Y(x,X)$
indexed by $0, i, j$ becomes
\[
\begin{pmatrix}
    1 & x_i & x_i \\
    x_i & X_{ii} & x_i \\
    x_i & x_i & X_{jj}
\end{pmatrix} \succeq 0.
\]
Positive semidefiniteness gives $X_{ii} X_{jj} \ge x_i^2$, while the
diagonal inequalities give $X_{ii} \le x_i$ and $X_{jj} \le x_i$. Together, $X_{ii}
= X_{jj} = x_i$. Now let $v := e_i - e_j \in \mathbb{R}^{n+1}$, where
$e_i$ and $e_j$ are standard basis vectors in the index space of
$Y$. Then
\[
    v^T Y v = X_{ii} - 2 X_{ij} + X_{jj}
    = x_i - 2 x_i + x_i = 0.
\]
Since $Y \succeq 0$, this implies $Yv = 0$, which gives the result.
\end{proof}

\begin{proof}[Proof of Proposition \ref{pro:at_least_one_strictly_lower}]
By assumption, after sorting by $x$, there exists $i < j$ such that
the s-lower inequality $X_{ij} \le x_j$ is active. By
Lemma~\ref{lem:lower_active_duplicate}, the columns of $Y(x,X)$
indexed by $i$ and $j$ are identical.

Define the $(n-1)$-dimensional data $(\hat{Q}, \hat{c})$ by merging
indices $i$ and $j$. In particular, set $\hat{Q}_{ii} := Q_{ii} +
Q_{jj} + 2Q_{ij}$, $\hat{Q}_{ik} := Q_{ik} + Q_{jk}$ for $k \ne i,j$,
and other entries of $\hat{Q}$ unchanged. Also set $\hat{c}_i := c_i +
c_j$, while keeping other entries of $\hat{c}$ unchanged. Then delete
index~$j$ throughout. Since the off-diagonal entries of $Q$ are
nonpositive, $(\hat{Q}, \hat{c})$ is submodular. Let $(\hat{x},
\hat{X})$ denote the reduced pair obtained from $(x,X)$ by deleting
index~$j$. Then $(\hat{x}, \hat{X})$ is feasible for the
$(n-1)$-dimensional SDP relaxation with data $(\hat{Q}, \hat{c})$. In
particular, $Y(\hat{x},\hat{X})$ is a principal submatrix of $Y(x,X)
\succeq 0$, and $\hat{X} \le \hat{x} e^T$ is inherited from $X \le
xe^T$, where abusing notation $e$ is the vector of all ones of the
appropriate size in each case.

Let $\rho_n$ and $\rho_{n-1}$ denote the optimal SDP values in dimensions
$n$ and $n-1$ for data $(Q,c)$ and $(\hat{Q}, \hat{c})$, respectively.
Also let $v_n$ and $v_{n-1}$ denote the corresponding optimal values
of the underlying nonconvex QPs.  The duplicate structure ensures
$\hat{Q} \bullet \hat{X} + \hat{c}^T \hat{x} = Q \bullet X + c^T x =
\rho_n$, so $\rho_{n-1} \le \rho_n$. Since the reduced QP is equivalent to
minimizing $x^T Q x + c^T x$ over $[0,1]^n$ with $x_i = x_j$, we have
$v_{n-1} \ge v_n$. By the induction hypothesis, $\rho_{n-1} = v_{n-1}$.
Therefore,
\[
\rho_n \le v_n \le v_{n-1} = \rho_{n-1} \le \rho_n,
\]
and all inequalities are equalities. In particular, $\rho_n = v_n$, so the
SDP relaxation is tight.
\end{proof}

\subsection{\textcolor{black}{Verification of the counterexample}} \label{sec:counterexample}

\textcolor{black}{This section verifies Example \ref{exa:counter} for the
instance (\ref{equ:counter}). Throughout, all arithmetic is exact and
carried out over the rational numbers, except where a floating-point
solver is used.}

\begin{proposition} \label{pro:counter}
\textcolor{black}{For the data $(Q,c)$ of (\ref{equ:counter}), problem
(\ref{equ:qpb}) has optimal value $0$, attained uniquely at $x = 0$, while the
relaxation (\ref{equ:sdp}) has optimal value at most $-109/1024$. In
particular, the relaxation gap is positive.}
\end{proposition}

\begin{proof}
\textcolor{black}{Every global minimizer of (\ref{equ:qpb}) is a KKT
point, and every KKT point has an active-set pattern assigning each
variable to $\{0\}$, $\{1\}$, or the open interval $(0,1)$. Enumerating
all $3^4 = 81$ patterns and solving each stationarity system over the
rationals is therefore exhaustive. Four of the $81$ reduced Hessians are
singular, but in each of those cases the stationarity system is
inconsistent, so no face carries a degenerate set of stationary
points.\footnote{\textcolor{black}{Had such a set existed, it would have been
harmless: on an affine set of stationary points $f$ is constant, since a null
direction $v$ of the reduced Hessian satisfies both $\nabla f \cdot v = 0$ and
$v^T Q v = 0$.}} The minimum over all KKT points equals $0$ and is attained
only at $x = 0$.}

\textcolor{black}{For the relaxation it suffices to exhibit a single feasible
point of negative objective value. Take}
\[
\textcolor{black}{
  \bar{x} = \begin{pmatrix} 3/32 \\ 3/16 \\ 9/16 \\ 3/4 \end{pmatrix},
  \qquad
  \bar{X} = \begin{pmatrix}
    3/32 & 3/32 & 3/32 & 61/1024 \\
    3/32 & 125/1024 & 3/16 & 3/16 \\
    3/32 & 3/16 & 231/512 & 9/16 \\
    61/1024 & 3/16 & 9/16 & 3/4
  \end{pmatrix}. }
\]
\textcolor{black}{All sixteen inequalities $\bar{X}_{ij} \le \bar{x}_i$ hold, and
the leading principal minors of $Y(\bar{x},\bar{X})$ are
all strictly positive. Hence $Y(\bar{x},\bar{X}) \succ 0$ by
Sylvester's criterion, and $(\bar{x},\bar{X})$ is feasible for
(\ref{equ:sdp}). Its objective value is exactly $c^T\bar{x} + Q \bullet
\bar{X} = -109/1024$.}
\end{proof}

\textcolor{black}{Solving (\ref{equ:sdp}) numerically with Mosek
\citep{mosek} gives an optimal value of approximately $-0.1220566$, and the
computed optimal solution satisfies $\rank(Y(x,X)) = 3$. Its vector
$x$ is already increasing, so no sorting is required, and exactly seven
of the sixteen inequalities $X_{ij} \le x_i$ are active: the first and
last diagonal inequalities, and five of the six s-upper inequalities,
the pair $(1,4)$ being the only slack one. In the notation of Section
\ref{sec:tight},}
\[
\textcolor{black}{ l = 0, \qquad d = 2, \qquad u = 5, \qquad a = l + d + u = 7, }
\]
\textcolor{black}{and the counting bound of Lemma \ref{lem:pataki}
is comfortably satisfied, since $r(r+1) = 12 \le 2(a+1) = 16$. This
configuration lies outside the hypotheses of every proposition used
in the proof of Theorem \ref{the:n=3}, and it is precisely the rank-3
configuration with partial activity anticipated in Remark \ref{rem:n4}.}


\textcolor{black}{It is natural to ask whether the gap can be closed by
strengthening (\ref{equ:sdp}). For $1 \le i < j < k \le n$ the triangle
inequalities}
\begin{equation} \label{equ:tri}
\textcolor{black}{
\begin{aligned}
  X_{ij} + X_{ik} &\le x_i + X_{jk}, &\qquad
  X_{ij} + X_{jk} &\le x_j + X_{ik}, \\
  X_{ik} + X_{jk} &\le x_k + X_{ij}, &\qquad
  x_i + x_j + x_k &\le X_{ij} + X_{ik} + X_{jk} + 1
\end{aligned} }
\end{equation}
\textcolor{black}{are valid for the Boolean quadric polytope
\citep{Padberg}, and hence valid on $\Cc\Pc(\Jc)$ as well; see
\cite{Burer_Letchford_2009}, extending \cite{Yajima_Fujie_1998}.
Table \ref{tab:cutladder} reports the optimal value of (\ref{equ:sdp})
strengthened successively by the full RLT inequalities and by all sixteen
triangle inequalities available at $n = 4$.}

\begin{table}[!htbp]
\begin{center}
\textcolor{black}{%
\begin{tabular}{lr}
\hline
Relaxation & Optimal value \\
\hline
PSD ${}+{}$ RLT upper bounds, i.e.\ (\ref{equ:sdp}) & $-0.1220566$ \\
PSD ${}+{}$ full RLT & $-0.1220566$ \\
PSD ${}+{}$ full RLT ${}+{}$ triangle inequalities & $-0.0001432$ \\
\hline
\end{tabular}}
\caption{\textcolor{black}{Optimal values of successively stronger relaxations
for the instance (\ref{equ:counter}). The true optimal value of
(\ref{equ:qpb}) is $0$. The triangle inequalities remove more than $99.9\%$ of
the gap but do not close it.}}
\label{tab:cutladder}
\end{center}
\end{table}

\textcolor{black}{We see that the gap is not an artifact of having discarded
the RLT lower bounds in (\ref{equ:sdp}); the point $(\bar{x},\bar{X})$ of
Proposition \ref{pro:counter} satisfies the lower bounds $X_{ij} \ge 0$
and $X_{ij} \ge x_i + x_j - 1$ as well, so the full SDP-RLT relaxation
of Proposition \ref{pro:BurerAnstreicher} attains the same value. In
addition, the triangle inequalities remove most, but not all, of the
gap. This surviving gap deserves an exact
certificate. Mixing the computed optimum with the strictly feasible point $x =
e/2$, $X = \tfrac14 ee^T + \tfrac18 I$ and rounding to denominator $10^8$
yields a rational point satisfying, exactly, the semidefiniteness condition,
all sixteen RLT inequalities in both directions, and all sixteen triangle
inequalities, with objective value}
\[
\textcolor{black}{
  -\frac{14113}{100{,}000{,}000} \;=\; -1.4113 \times 10^{-4} \;<\; 0 .}
\]
\textcolor{black}{Since the optimal value of (\ref{equ:qpb}) is $0$ by
Proposition \ref{pro:counter}, the relaxation strengthened by the full
RLT and triangle inequalities retains a positive gap on this instance.}

\textcolor{black}{Figure \ref{fig:gap2} depicts this instance in a
two-dimensional projection that makes both the size of the gap and the effect
of the triangle inequalities visible at once. Rather than
fixing marginals as in Figure \ref{fig:tightproj}, it projects the full bodies
onto the two affine functionals}
\[
\textcolor{black}{
  t(x,X) \;=\; c^Tx + Q \bullet X
  \qquad\text{and}\qquad
  g(x,X) \;=\; x_3 + X_{14} - X_{13} - X_{34} ,}
\]
\textcolor{black}{the second of which is the slack of one of the two violated
triangle inequalities. Both are affine, so the projection is linear and the
inclusion $\Cc\Pc(\Jc) \subseteq \Lc$ is preserved in the image. The
horizontal dimension is then self-interpreting: the leftmost point of each shadow
is the optimal value of the corresponding problem, so the horizontal offset
between the two tips is exactly the relaxation gap. Vertically, the exact hull
lies entirely above the line $g = 0$, since the triangle inequality is valid
on $\Cc\Pc(\Jc)$, whereas the relaxation protrudes below it.}

\begin{figure}[!htbp]
\begin{center}
\textcolor{black}{\includegraphics[width=0.92\textwidth]{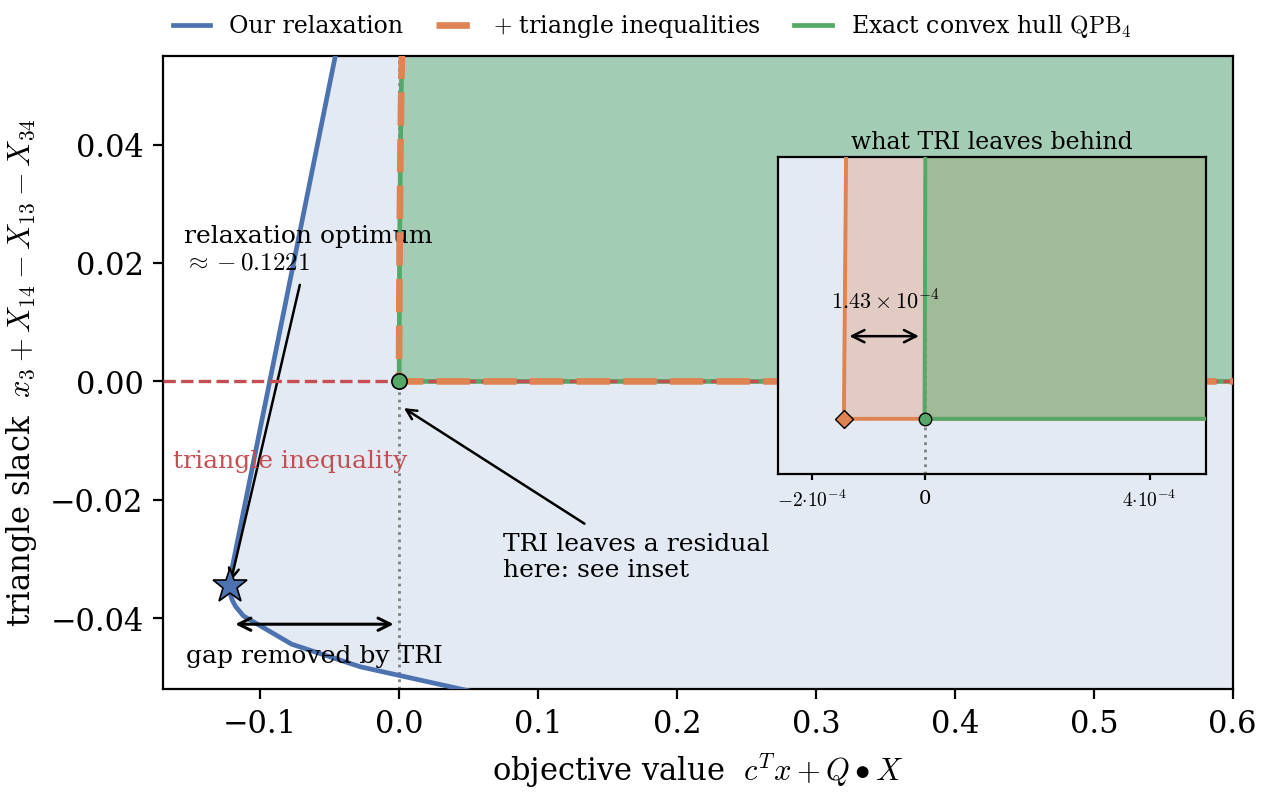}}
\caption{\textcolor{black}{The instance (\ref{equ:counter}) projected onto
(objective value, triangle-inequality slack). Adding the triangle inequalities
cuts the body at $g = 0$ and removes more than $99.9\%$ of the gap, but not
all of it. Inset (magnified roughly $1000\times$): the strengthened relaxation
still reaches $-1.43 \times 10^{-4}$, whereas the exact hull stops at $0$.}}
\label{fig:gap2}
\end{center}
\end{figure}

\section{Conclusions} \label{sec:conclusion}

\textcolor{black}{We now return to the complexity landscape summarized
in Table~\ref{tab:complexity} of the Introduction, characterizing the
complexity of minimizing a general quadratic function $f(x) := x^T Q x
+ c^T x $ over the sets $\{0,1\}^n$ and $[0,1]^n$. In particular, the
polynomial-time solvable submodular case over $\{0,1\}^n$ is well-known
as discussed in the Introduction, while the complexity over $[0,1]^n$
has not been fully established. Two special cases where polynomial time
solvability is known over $[0,1]^n$ is when additional restrictions are
imposed on the signs of the diagonal entries of $Q$ or the entries of
$c$. This paper adds to this landscape by establishing new tractability
for the submodular case over $[0,1]^n$ when $n \le 3$ by removing
these additional restrictions. Example \ref{exa:counter} delimits that
contribution precisely. The relaxation (\ref{equ:sdp}) is not tight
for $n \ge 4$, and the gap it exhibits is closed neither by the RLT
lower bounds nor by the triangle inequalities. What fails at $n \ge 4$
is therefore this particular semidefinite certificate, not necessarily
tractability itself. The complexity of minimizing a submodular quadratic
over $[0,1]^n$ for $n \ge 4$ remains open, and settling it will require
either a different certificate or a hardness reduction. The experiments
of Section \ref{sec:empirical} indicate that instances exhibiting a gap
are highly non-generic, so the relaxation continues to perform well in
practice even though it is not tight in general.}


\textcolor{black}{Finally, it is natural to ask whether sparsity in $Q$ can be
exploited to simplify the relaxation (\ref{equ:sdp}). Since the entry
$X_{ij}$ enters the objective only through $Q_{ij}$, the RLT upper
bound $X_{ij} \le \min\{x_i,x_j\}$ is intuitively unnecessary whenever
$Q_{ij}=0$. When $Q$ is block separable, this holds exactly: the
problem decomposes across blocks, and only the within-block bounds are
needed. For general sparse $Q$, our experiments indicate that dropping
the bounds associated with zero entries of $Q$ preserves the optimal
value in the large majority of instances, though not always. A precise
characterization of when such bounds are redundant---and hence of the
smallest sufficient set of RLT inequalities---is a promising direction
for future work.}


\subsection*{\textcolor{black}{Acknowledgements}}
\textcolor{black}{We thank Shaoning Han for pointing out an error in an
earlier version of this paper.}

\end{onehalfspace}

\bibliographystyle{abbrvnat}
\bibliography{main}

\end{document}